\documentclass[pdflatex,sn-mathphys-num]{sn-jnl}
\usepackage{graphicx} 
\usepackage{bm}
\usepackage{mathtools,amssymb}
\usepackage{physics}
\usepackage{mleftright}
\usepackage{chngcntr}
\counterwithin{footnote}{section}

\numberwithin{equation}{section}
\usepackage{doi}
\usepackage{amsthm}
\theoremstyle{thmstyleone}
\newtheorem{thm}{Theorem}[section]
\newtheorem{cor}[thm]{Corollary}
\newtheorem{lem}[thm]{Lemma}
\newtheorem{prop}[thm]{Proposition}
\theoremstyle{thmstyletwo}

\newtheorem{rem}{Remark}[section]
\theoremstyle{thmstylethree}
\newtheorem{dfn}{Definition}[section]
\newcommand{\Iwade}{Smooth-Convex-Concave Splitting Scheme}
\newcommand{\set}[2]{\left\{\,#1\,\middle|\,#2\,\right\}}
\newcommand{\cc}{c}
\newcommand{\cx}{\cc_x}
\newcommand{\cy}{\cc_y}
\newcommand{\cz}{\cc_z}
\newcommand{\bigkakko}[3]{\left[#1\right]_{#2}^{#3}}
\newcommand{\delpi}{{\delta_{i}^{+}}}
\newcommand{\delpj}{{\delta_{j}^{+}}}
\newcommand{\delpk}{{\delta_{k}^{+}}}
\newcommand{\delpifunc}[1]{\delta_{i}^{+}\mleft(#1\mright)}
\newcommand{\delpjfunc}[1]{\delta_{j}^{+}\mleft(#1\mright)}
\newcommand{\delpkfunc}[1]{\delta_{k}^{+}\mleft(#1\mright)}
\newcommand{\delmi}{\delta_{i}^{-}}
\newcommand{\delmj}{\delta_{j}^{-}}
\newcommand{\delmk}{\delta_{k}^{-}}
\newcommand{\delmifunc}[1]{\delta_{i}^{-}\mleft(#1\mright)}
\newcommand{\delmjfunc}[1]{\delta_{j}^{-}\mleft(#1\mright)}
\newcommand{\delmkfunc}[1]{\delta_{k}^{-}\mleft(#1\mright)}
\newcommand{\deli}{\delta_{i}^{\langle 1\rangle}}
\newcommand{\delj}{\delta_{j}^{\langle 1\rangle}}
\newcommand{\delk}{\delta_{k}^{\langle 1\rangle}}

\newcommand{\delii}{{\delta_{i}^{\langle 2\rangle}}}
\newcommand{\deljj}{{\delta_{j}^{\langle 2\rangle}}}
\newcommand{\delkk}{{\delta_{k}^{\langle 2\rangle}}}
\newcommand{\deltwo}{{\Delta_{\mathrm{d}}}}
\newcommand{\delfour}{{\deltwo^{2}}}
\newcommand{\Dt}{\Delta t}
\newcommand{\Dx}{\Delta x}
\newcommand{\Dy}{\Delta y}
\newcommand{\Dz}{\Delta z}
\newcommand{\LL}{L}
\newcommand{\Lx}{\ell_x}
\newcommand{\Ly}{\ell_y}
\newcommand{\Lz}{\ell_z}
\newcommand{\NN}{N}
\newcommand{\Nx}{\NN_x}
\newcommand{\Ny}{\NN_y}
\newcommand{\Nz}{\NN_z}
\newcommand{\aaa}{0}
\newcommand{\ax}{0}
\newcommand{\ay}{0}
\newcommand{\az}{0}
\newcommand{\bb}{\NN}
\newcommand{\bx}{\Nx}
\newcommand{\by}{\Ny}
\newcommand{\bz}{\Nz}
\newcommand{\vv}{v}
\newcommand{\vfunc}[3]{\vv\mleft(#1,#2,#3\mright)}
\newcommand{\uu}{u}
\newcommand{\ufunc}[4]{\uu\mleft(#1,#2,#3,#4\mright)}
\newcommand{\uijkn}[4]{{\uu_{{#1},{#2},{#3}}^{(#4)}}}
\newcommand{\dudtn}[1]{{\left.\pdv{\uu}{t}\right|^{(#1)}}}
\newcommand{\dudtijkn}[4]{{\left.\pdv{\uu}{t}\right|_{{#1},{#2},{#3}}^{(#4)}}}
\newcommand{\Deltaun}[1]{{\left.\Delta \uu\right|^{(#1)}}}
\newcommand{\Deltauijkn}[4]{{\left.\Delta \uu\right|_{{#1},{#2},{#3}}^{(#4)}}}
\newcommand{\Deltatwoun}[1]{{\left.\Delta^{2}\uu\right|^{(#1)}}}
\newcommand{\Deltatwouijkn}[4]{{\left.\Delta^{2}\uu\right|_{{#1},{#2},{#3}}^{(#4)}}}
\newcommand{\UU}{U}
\newcommand{\Ubar}{\bar{\UU}}
\newcommand{\Util}{\tilde{\UU}}
\newcommand{\Uhat}{\hat{\UU}}
\newcommand{\Ubarijk}[3]{\Ubar_{{#1},{#2},{#3}}}
\newcommand{\Uijk}[3]{\UU_{{#1},{#2},{#3}}}
\newcommand{\Uijkn}[4]{{\UU_{{#1},{#2},{#3}}^{(#4)}}}
\newcommand{\eijkn}[4]{{e_{{#1},{#2},{#3}}^{(#4)}}}
\newcommand{\Un}[1]{{\UU^{(#1)}}}
\newcommand{\Um}[1]{{\UU^{\llbracket#1\rrbracket}}}
\newcommand{\Utiln}[1]{{\Util^{(#1)}}}
\newcommand{\Uhatn}[1]{{\Uhat^{(#1)}}}
\newcommand{\un}[1]{{\uu^{(#1)}}}
\newcommand{\unn}[1]{{\uu^{[#1]}}}
\newcommand{\unnfunc}[4]{{\uu^{[#4]}\mleft(#1,#2,#3\mright)}}
\newcommand{\en}[1]{{e^{(#1)}}}
\newcommand{\fF}{f}

\newcommand{\fstar}{\fF^{*}}
\newcommand{\gG}{g}

\newcommand{\hhhh}{h}
\newcommand{\fii}[1]{{\fF_{#1}}}

\newcommand{\fijk}[3]{{\fF_{{#1},{#2},{#3}}}}
\usepackage{stmaryrd}

\newcommand{\gijk}[3]{{\gG_{{#1},{#2},{#3}}}}
\newcommand{\deliUU}[1]{\left.\delta_{i}^{\langle 1\rangle}\Uijk{i}{j}{k}\right|_{i=#1}}
\newcommand{\deljUU}[1]{\left.\delta_{j}^{\langle 1\rangle}\Uijk{i}{j}{k}\right|_{j=#1}}
\newcommand{\delkUU}[1]{\left.\delta_{k}^{\langle 1\rangle}\Uijk{i}{j}{k}\right|_{k=#1}}
\newcommand{\deliU}[1]{\left.\delta_{i}^{\langle 1\rangle}\Uijkn{i}{j}{k}{n}\right|_{i=#1}}
\newcommand{\deljU}[1]{\left.\delta_{j}^{\langle 1\rangle}\Uijkn{i}{j}{k}{n}\right|_{j=#1}}
\newcommand{\delkU}[1]{\left.\delta_{k}^{\langle 1\rangle}\Uijkn{i}{j}{k}{n}\right|_{k=#1}}
\newcommand{\deliu}[1]{\left.\delta_{i}^{\langle 1\rangle}\uijkn{i}{j}{k}{n}\right|_{i=#1}}
\newcommand{\delju}[1]{\left.\delta_{j}^{\langle 1\rangle}\uijkn{i}{j}{k}{n}\right|_{j=#1}}
\newcommand{\delku}[1]{\left.\delta_{k}^{\langle 1\rangle}\uijkn{i}{j}{k}{n}\right|_{k=#1}}
\newcommand{\delidelijkUU}[1]{\left.\delta_{i}^{\langle 1\rangle}\deltwo\Uijk{i}{j}{k}\right|_{i=#1}}
\newcommand{\deljdelijkUU}[1]{\left.\delta_{j}^{\langle 1\rangle}\deltwo\Uijk{i}{j}{k}\right|_{j=#1}}
\newcommand{\delkdelijkUU}[1]{\left.\delta_{k}^{\langle 1\rangle}\deltwo\Uijk{i}{j}{k}\right|_{k=#1}}
\newcommand{\delidelijkU}[1]{\left.\delta_{i}^{\langle 1\rangle}\deltwo\Uijkn{i}{j}{k}{n}\right|_{i=#1}}
\newcommand{\deljdelijkU}[1]{\left.\delta_{j}^{\langle 1\rangle}\deltwo\Uijkn{i}{j}{k}{n}\right|_{j=#1}}
\newcommand{\delkdelijkU}[1]{\left.\delta_{k}^{\langle 1\rangle}\deltwo\Uijkn{i}{j}{k}{n}\right|_{k=#1}}
\newcommand{\delidelijku}[1]{\left.\delta_{i}^{\langle 1\rangle}\deltwo\uijkn{i}{j}{k}{n}\right|_{i=#1}}
\newcommand{\deljdelijku}[1]{\left.\delta_{j}^{\langle 1\rangle}\deltwo\uijkn{i}{j}{k}{n}\right|_{j=#1}}
\newcommand{\delkdelijku}[1]{\left.\delta_{k}^{\langle 1\rangle}\deltwo\uijkn{i}{j}{k}{n}\right|_{k=#1}}
\newcommand{\sumd}[2]{{\sum_{#1}^{#2}}\,\!''}
\newcommand{\nabladp}{\nabla_{\mathrm{d}}^{+}}
\newcommand{\nabladpfunc}[1]{\nabla_{\mathrm{d}}^{+}\mleft(#1\mright)}
\newcommand{\nabladm}{\nabla_{\mathrm{d}}^{-}}
\newcommand{\nabladmfunc}[1]{\nabla_{\mathrm{d}}^{-}\mleft(#1\mright)}
\newcommand{\normLd}[2]{\norm{#2}_{L_{\mathrm{d}}^{#1}}}
\newcommand{\DD}{\mathrm{D}}
\newcommand{\Ldtwo}{L_{\mathrm{d}}^{2}}
\newcommand{\productLdtwo}[2]{\left\langle\, #1\, ,\,#2\,\right\rangle_{\Ldtwo}}
\newcommand{\normLdtwo}[1]{\normLd{2}{#1}}
\newcommand{\nablaLdtwo}{\nabla_{\Ldtwo}}
\newcommand{\nuu}{\nu}
\newcommand{\nufunc}[1]{\nuu\mleft(#1\mright)}
\newcommand{\nuone}{\nuu_{1}}
\newcommand{\nuonefunc}[1]{\nuone\mleft(#1\mright)}
\newcommand{\nutwo}{\nuu_{2}}
\newcommand{\nutwofunc}[1]{\nutwo\mleft(#1\mright)}
\newcommand{\nuthree}{\nuu_{3}}
\newcommand{\nuthreefunc}[1]{\nuthree\mleft(#1\mright)}
\newcommand{\nufourn}[1]{\rho^{(#1)}}
\newcommand{\nufournfunc}[2]{\nufourn{#1}\mleft(#2\mright)}
\newcommand{\nusum}{\nuu_{\mathrm{total}}}
\newcommand{\nusumfunc}[1]{\nuu_{\mathrm{total}}\mleft(#1\mright)}
\newcommand{\nutilone}{\nabla_{L^{2}}\tilde{\nuu}_{1}}
\newcommand{\nutilonefunc}[1]{\nutilone\mleft(#1\mright)}
\newcommand{\nutiltwo}{\nabla_{L^{2}}\tilde{\nuu}_{2}}
\newcommand{\nutiltwofunc}[1]{\nutiltwo\mleft(#1\mright)}
\newcommand{\nutilthree}{\nabla_{L^{2}}\tilde{\nuu}_{3}}
\newcommand{\nutilthreefunc}[1]{\nutilthree\mleft(#1\mright)}
\newcommand{\phione}{\phi_{1}}
\newcommand{\phionefunc}[1]{\phione\mleft(#1\mright)}
\newcommand{\phitwo}{\phi_{2}}
\newcommand{\phitwofunc}[1]{\phitwo\mleft(#1\mright)}
\newcommand{\phithree}{\phi_{3}}
\newcommand{\phithreefunc}[1]{\phithree\mleft(#1\mright)}
\newcommand{\phifour}{\phi_{4}}
\newcommand{\phifourfunc}[1]{\phifour\mleft(#1\mright)}
\newcommand{\phionetilde}{\tilde{\phi}_{1}}
\newcommand{\phionetildefunc}[1]{\phionetilde\mleft(#1\mright)}
\newcommand{\phionehat}{\hat{\phi}_{1}}
\newcommand{\phionehatfunc}[1]{\phionehat\mleft(#1\mright)}
\newcommand{\psii}{\psi}
\newcommand{\psifunc}[1]{\psi\mleft(#1\mright)}
\newcommand{\psidfunc}[1]{\psi'\mleft(#1\mright)}
\newcommand{\psiddfunc}[1]{\psi''\mleft(#1\mright)}
\newcommand{\psihat}{\hat{\psii}}
\newcommand{\psihatfunc}[1]{\psihat\mleft(#1\mright)}
\newcommand{\xii}{\xi}
\newcommand{\xiijk}[3]{\xii_{{#1},{#2},{#3}}}
\newcommand{\Hfunc}[1]{H\mleft(#1\mright)}
\newcommand{\Htruefunc}[1]{H_{\mathrm{true}}\mleft(#1\mright)}
\newcommand{\Htilfunc}[1]{\tilde{H}\mleft(#1\mright)}
\newcommand{\zetaa}{\zeta}
\newcommand{\CC}{C}
\newcommand{\CCd}{C'}
\newcommand{\Ceez}{\CC_{\varepsilon,\eta,\zetaa}}
\newcommand{\CT}{\CC_{T}}
\newcommand{\CTd}{{\CC'_{T}}}
\newcommand{\CTdd}{{\CC''_{T}}}
\newcommand{\omegaa}{\omega}
\newcommand{\thetafunc}[1]{\theta\mleft(#1\mright)}
\newcommand{\nd}{n'}

\newcommand{\kappan}[1]{\kappa_{#1}}
\newcommand{\Lhat}{\hat{\LL}}
\newcommand{\rr}{r}
\newcommand{\iotaa}{\iota}
\newcommand{\iotaafunc}[1]{\iotaa\mleft(#1\mright)}
\newcommand{\iotaafuncijk}[4]{\left\{\iotaa\mleft(#1\mright)\right\}_{{#2},{#3},{#4}}}
\newcommand{\chii}{\chi}
\newcommand{\chimfunc}[2]{\chii^{\llbracket#1\rrbracket}\mleft(#2\mright)}
\newcommand{\gamman}[1]{{\gamma^{(#1)}}}
\newcommand{\gammatiln}[1]{{\tilde{\gamma}^{(#1)}}}

\newcommand{\Lambdaa}{\Lambda}

\newcommand{\mm}{m}

\newcommand{\Md}{{M'}}
\newcommand{\placeholder}{\,\cdot\,}
\makeatletter
\gdef\orcidlogo{\raisebox{0ex}{\includegraphics[height=2ex]{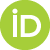}}}
\makeatother

\begin{document}

\title[A Mathematical Analysis of a \Iwade \ for the Swift--Hohenberg Equation]{A Mathematical Analysis of a \Iwade \ for the Swift--Hohenberg Equation}

\author*[1]{\fnm{Yuki} \sur{Yonekura}}\email{yonekurayuki00@gmail.com}

\author[1]{\fnm{Daiki} \sur{Iwade}}\email{iwadedaiki@gmail.com}

\author[2]{\fnm{Shun} \sur{Sato}~\orcid{https://orcid.org/0000-0002-8938-0825}}\email{shun\_sato@tmu.ac.jp}

\author[1]{\fnm{Takayasu} \sur{Matsuo}~\orcid{https://orcid.org/0009-0007-1336-3934}}\email{matsuo@mist.i.u-tokyo.ac.jp}

\affil*[1]{\orgdiv{Department of Mathematical Informatics}, \orgdiv{Graduate School of Information Science and Technology}, \orgname{The University of Tokyo}, \orgaddress{\street{7-3-1 Hongo, Bunkyo-ku}, \city{Tokyo}, \postcode{113-8656}, \country{Japan}}}

\affil[2]{\orgdiv{Department of Mathematical Sciences}, \orgdiv{Graduate School of Science}, \orgname{Tokyo Metropolitan University}, \orgaddress{\street{1-1 Minami-Osawa, Hachioji-shi}, \city{Tokyo}, \postcode{192-0397}, \country{Japan}}}

\abstract{
    The Swift--Hohenberg equation is a widely studied fourth-order model, originally proposed to describe hydrodynamic fluctuations. It admits an energy-dissipation law and, under suitable assumptions, bounded solutions. Many structure-preserving numerical schemes have been proposed to retain such properties; however, existing approaches are often fully implicit and therefore computationally expensive. We introduce a simple design principle for constructing dissipation-preserving finite difference schemes and apply it to the Swift--Hohenberg equation in three spatial dimensions. Our analysis relies on discrete inequalities for the underlying energy, assuming a Lipschitz continuous gradient and either convexity or $\mu$-strong convexity of the relevant terms. The resulting method is linearly implicit, yet it preserves the original energy-dissipation law, guarantees unique solvability, ensures boundedness of numerical solutions, and admits an a priori error estimate, provided that the time step is sufficiently small. To the best of our knowledge, this is the first linearly implicit finite difference scheme for the Swift--Hohenberg equation for which all of these properties are established.
}


\maketitle

\section{Introduction}
\label{sec:Introduction}
The Swift--Hohenberg equation was introduced in~\cite{SH} as a model for hydrodynamic fluctuations. On a domain $\Omega\coloneqq[0,\Lx]\times[0,\Ly]\times[0,\Lz]$, the equation takes the form
\begin{equation}
    \label{equ:SH}
    \pdv{\uu}{t}=-\left(\uu^{3}+(1-\eta)\uu+\varepsilon\Delta^{2}\uu+2\Delta\uu\right),
\end{equation}
where $\uu:\Omega\times[0,\infty)\to\mathbb{R}$ denotes the dependent variable, $\varepsilon>0$ and $\eta\in\mathbb{R}$ are constants, and $ \Delta $ denotes the Laplacian operator. The boundary conditions are
\[
    \begin{dcases}
        \nabla\uu\cdot\bm{n}=0 & \text{on}\ \partial\Omega,\\
        \nabla(\Delta\uu)\cdot\bm{n}=0 & \text{on}\ \partial\Omega,
    \end{dcases}
\]
where $\bm{n}$ is the outward unit normal vector on $\partial\Omega$. It has been applied in various fields, such as fluid mechanics~\cite{SH}, optical physics~\cite{SH_appli_1}, and neuroscience~\cite{SH_appli_2}.
\par The Swift--Hohenberg equation admits the following structural properties. For the associated energy functional
\begin{equation}
    \label{equ:SH_energy}
    \Htruefunc{\uu}\coloneqq\int_{\Omega}\left(\frac{1}{4}\uu^{4}+\frac{1-\eta}{2}\uu^{2}+\frac{\varepsilon}{2}(\Delta \uu)^{2}-|\nabla\uu|^{2}\right)\dd{\bm{x}},
\end{equation}
the Swift--Hohenberg equation satisfies an energy-dissipation law 
\[
    \dv{\Htruefunc{\uu}}{t}\leq 0.
\]
Moreover, classical solutions remain bounded in time (see Section~\ref{subsec:Properties_of_the_Swift--Hohenberg_equation}). Hence, it is desirable to design a numerical scheme which inherits these properties at the discrete level. 
\par There have been many studies on numerical schemes for the Swift--Hohenberg equation. In~\cite{senko_1}, a numerical scheme for the one-dimensional case was proposed and validated by numerical experiments, but no theoretical guarantee of energy-dissipation or boundedness was provided. In~\cite{senko_2,senko_2_2,senko_2_3}, energy-stable schemes for two- and three-dimensional cases were proposed; however, these schemes require solving a nonlinear system at each time step, and a proof of boundedness of the numerical solutions is not addressed there. In~\cite{senko_3}, a space-time discretization that strictly respects a Lyapunov functional was introduced, emphasizing energy stability; however, boundedness or unique solvability are not addressed there.
\par Recently, several approaches have been developed to establish additional properties beyond energy stability. In~\cite{senko_4}, a fully implicit scheme for two- and three-dimensional cases was developed, and it was proved to satisfy the energy-dissipation law, boundedness, and an error estimate. In~\cite{senko_5}, a linearly implicit approach based on a stabilized Lagrange multiplier was introduced, together with error estimates; this scheme satisfies a modified energy-dissipation law and involves an auxiliary parameter. In~\cite{senko_FE}, a finite element scheme for the Swift--Hohenberg equation was proposed, establishing unique solvability, an energy-dissipation property, and boundedness; however, the error analysis is left open. In~\cite{Barbara}, a semi-discretized linearly implicit scheme for the Swift--Hohenberg equation was introduced, and it was shown that this scheme satisfies an energy-dissipation law; however, an a priori error estimate is not established there.
\par It is also possible to adapt the finite volume method to the Swift--Hohenberg equation. Indeed,~\cite{senko_notSH_FV_1} shows that for the biharmonic problem one can construct finite volume schemes with convergence and error analysis on general meshes, and~\cite{senko_notSH_FV_2} proposes bound-preserving and energy-dissipating finite volume schemes for the Cahn--Hilliard equation. However, implementing finite volume schemes for fourth-order models is generally involved, and in this sense, a finite difference scheme is often more practical. 
\par Therefore, it is important to develop a numerical scheme that reduces computational cost while still satisfying the energy-dissipation law, boundedness, unique solvability, and an error estimate. 
\par Our approach is based on a simple design principle, motivated by the splitting idea in~\cite{Iwade}. Specifically, we construct energy-dissipating schemes by appropriately splitting the right-hand side of the partial differential equation (PDE) into explicit and implicit terms. Building on that idea, we propose a new linearly implicit scheme for the Swift--Hohenberg equation.
\par The proposed scheme has the following properties. It is linearly implicit, preserves the original energy-dissipation law, is uniquely solvable at each time step, ensures boundedness of numerical solutions, and admits an a priori error estimate. These properties are proved under the assumption that the time step is smaller than or equal to a constant independent of the mesh size. To the best of our knowledge, this is the first linearly implicit finite difference scheme for the Swift--Hohenberg equation for which all of the above properties are established.
\par When applying the general principle to the Swift--Hohenberg equation (Section~\ref{sec:A_Dissipation-Preserving_Scheme_for_the_Swift--Hohenberg_Equation}), we work in three spatial dimensions with homogeneous Neumann boundary conditions. Analogous statements hold for periodic boundary conditions, and the one- or two-dimensional cases are similar.
\bmhead{Organization} Section~\ref{sec:Preliminaries} collects notation and standard preliminaries for the analysis. Section~\ref{sec:Dissipation-Preserving_Scheme_Framework} presents a design principle for dissipation-preserving schemes, establishes unique solvability and a discrete energy-dissipation law, and derives an error estimate, under a time-step restriction and a suitable regularity assumption on the exact solution. Section~\ref{sec:A_Dissipation-Preserving_Scheme_for_the_Swift--Hohenberg_Equation} applies the principle to the Swift--Hohenberg equation and establishes its properties, including boundedness of numerical solutions.
\section{Preliminaries}
\label{sec:Preliminaries}
\subsection{Definitions of the Computational Domain and Difference Operators}
In this paper, we apply the finite difference method. The computational domain is a rectangular box with lengths $\Lx$, $\Ly$, and $\Lz$ in the $x$-, $y$-, and $z$-directions. The mesh sizes are denoted by $\Dx$, $\Dy$, and $\Dz$, and the domain is discretized into $\Nx+1$, $\Ny+1$, and $\Nz+1$ grid points in the $x$-, $y$-, and $z$-directions, respectively. Namely, $\Lx=\Nx\Dx$, $\Ly=\Ny\Dy$, and $\Lz=\Nz\Dz$. 
\par We define $\delpi, \delmi, \deli$ as the forward difference operator, the backward difference operator, and the central difference operator for the $x$-direction, respectively, i.e.,
\[
    \begin{split}
        \delpi\fijk{i}{j}{k}&\coloneqq\frac{\fijk{i+1}{j}{k}-\fijk{i}{j}{k}}{\Dx},\\
        \delmi\fijk{i}{j}{k}&\coloneqq\frac{\fijk{i}{j}{k}-\fijk{i-1}{j}{k}}{\Dx},\\
        \deli\fijk{i}{j}{k}&\coloneqq\frac{\fijk{i+1}{j}{k}-\fijk{i-1}{j}{k}}{2\Dx}.\\
    \end{split}
\]
Similarly, $\delpj$, $\delmj$, and $\delj$ denote the corresponding operators for the $y$-direction, and $\delpk$, $\delmk$, and $\delk$ for the $z$-direction. We also define $\nabladp$ and $\nabladm$ as discrete approximations of the gradient, and $\deltwo$ as a discrete approximation of the Laplacian, i.e.,
\[
    \nabladp\coloneqq\left(\begin{array}{c}
        \delpi\\
        \delpj\\
        \delpk
    \end{array}\right),\quad\nabladm\coloneqq\left(\begin{array}{c}
        \delmi\\
        \delmj\\
        \delmk
    \end{array}\right),\quad\deltwo\coloneqq\nabladp\cdot\nabladm.
\]
\par The following lemma holds for these operators, which can be verified by a straightforward calculation.
\begin{lem}
    \label{lem:integration_by_parts}\cite{Furihata}
    For any three-dimensional grid functions $\fijk{i}{j}{k}$, and $\gijk{i}{j}{k}$, if
    \[
        \begin{dcases}
            \deli\fijk{\ax}{j}{k}=\deli\gijk{\ax}{j}{k}=\deli\fijk{\bx}{j}{k}=\deli\gijk{\bx}{j}{k}=0,\\
            \delj\fijk{i}{\ay}{k}=\delj\gijk{i}{\ay}{k}=\delj\fijk{i}{\by}{k}=\delj\gijk{i}{\by}{k}=0,\\
            \delk\fijk{i}{j}{\az}=\delk\gijk{i}{j}{\az}=\delk\fijk{i}{j}{\bz}=\delk\gijk{i}{j}{\bz}=0,
        \end{dcases}
    \]
    then
    \[
        \begin{split}
            &\quad\sumd{i=\ax}{\bx}\sumd{j=\ay}{\by}\sumd{k=\az}{\bz}\fijk{i}{j}{k}\deltwo\gijk{i}{j}{k}\Dx\Dy\Dz\\
            &=-\sumd{i=\ax}{\bx}\sumd{j=\ay}{\by}\sumd{k=\az}{\bz}\frac{\nabladp\fijk{i}{j}{k}\cdot\nabladp\gijk{i}{j}{k}+\nabladm\fijk{i}{j}{k}\cdot\nabladm\gijk{i}{j}{k}}{2}\Dx\Dy\Dz,\\
        \end{split}
    \]
    where
    \[
        \sumd{i=\aaa}{\bb}\fii{i}\coloneqq\frac{\fii{\aaa}}{2}+\fii{1}+\dots+\fii{\bb-1}+\frac{\fii{\bb}}{2}.
    \]
\end{lem}
\subsection{Discrete Function Space and Convex Analysis}
\par We first define the discrete counterpart of the function space of functions $\Omega\to\mathbb{R}$ and discuss Lipschitz continuous gradients, convexity, and $\mu$-strong convexity in the discrete setting.
\par We define the vector space
\[
    \Ldtwo\coloneqq\set{\fF=\left(\fijk{i}{j}{k}\right)}{\fijk{i}{j}{k}\in\mathbb{R},i\in\left\{0,\dots,\Nx\right\},j\in\left\{0,\dots,\Ny\right\},k\in\left\{0,\dots,\Nz\right\}}
\]
and the inner product on it as follows: for $\fF,\gG\in\Ldtwo$, the inner product $\productLdtwo{\fF}{\gG}$ is given by
\[
    \productLdtwo{\fF}{\gG}\coloneqq\sumd{i=0}{\Nx}\sumd{j=0}{\Ny}\sumd{k=0}{\Nz}\fijk{i}{j}{k}\gijk{i}{j}{k}\Dx\Dy\Dz.
\]
In addition, we define the norms and the seminorm on it as follows: for $\fF\in\Ldtwo$, the norms $\normLdtwo{\fF}$ and $\normLd{\infty}{\fF}$, and the seminorm $\normLdtwo{\DD\fF}$ are given by
\[
    \begin{split}
        \normLdtwo{\fF}&\coloneqq\sqrt{\productLdtwo{\fF}{\fF}},\\
        \normLd{\infty}{\fF}&\coloneqq\max_{(i,j,k)\in\{0,\dots,N_{x}\}\times\{0,\dots,N_{y}\}\times\{0,\dots,N_{z}\}}\left|\fijk{i}{j}{k}\right|,\\
        \normLdtwo{\DD\fF}&\coloneqq\left(\sumd{i=0}{\Nx}\sumd{j=0}{\Ny}\sumd{k=0}{\Nz}\frac{\left|\nabladp\fijk{i}{j}{k}\right|^{2}+\left|\nabladm\fijk{i}{j}{k}\right|^{2}}{2}\Dx\Dy\Dz\right)^{\frac{1}{2}},
    \end{split}
\]
where, for any $c=\left(\begin{array}{c}
    \cx\\
    \cy\\
    \cz
\end{array}\right)\in\mathbb{R}^{3}$, we set
\[
    \abs{c}\coloneqq\sqrt{\cx^{2}+\cy^{2}+\cz^{2}}.
\]
Moreover, we introduce the mapping $\iotaa:\left\{\Omega\to\mathbb{R}\right\}\to\Ldtwo$ defined by
\[
    \iotaafuncijk{\vv}{i}{j}{k}\coloneqq\vfunc{i\Dx}{j\Dy}{k\Dz}
\]
for a function $\vv:\Omega\to\mathbb{R}$.
\par The following discussion in this subsection is based on~\cite{iwade_convex}.
\par For a function $\nuu:\Ldtwo\to\mathbb{R}$, we define the gradient of $\nuu$. Let $\nuu:\Ldtwo\to\mathbb{R}$ be a function and suppose that, for any $\fF,\hhhh\in\Ldtwo$,
\[
    \lim_{\Lambdaa\to +0}\frac{\nufunc{\fF+\Lambdaa\hhhh}-\nufunc{\fF}}{\Lambdaa}\quad(\Lambdaa\in\mathbb{R})
\]
is well-defined. Moreover, assume that the map
\[
    \Ldtwo\ni\hhhh\mapsto\lim_{\Lambdaa\to +0}\frac{\nufunc{\fF+\Lambdaa\hhhh}-\nufunc{\fF}}{\Lambdaa}\in\mathbb{R}
\]
is linear and continuous. Then, since $\Ldtwo$ is finite-dimensional and hence a Hilbert space, the Riesz representation theorem implies that there exists a unique $\nablaLdtwo\nufunc{\fF}\in\Ldtwo$ such that
\[
    \lim_{\Lambdaa\to +0}\frac{\nufunc{\fF+\Lambdaa\hhhh}-\nufunc{\fF}}{\Lambdaa}=\productLdtwo{\nablaLdtwo\nufunc{\fF}}{\hhhh}.
\]
Therefore, we define this $\nablaLdtwo\nufunc{\fF}\in\Ldtwo$ as the gradient of $\nuu$ at $\fF$.
\par A function $\nuu:\Ldtwo\to\mathbb{R}$ is called $\LL$-smooth if $\nablaLdtwo\nufunc{\fF}$ exists for all $\fF\in\Ldtwo$, and if $\nablaLdtwo\nufunc{\fF}$ is Lipschitz continuous; in other words, there exists a positive constant $\LL$ such that for all $\fF,\gG\in\Ldtwo$,
\[
    \normLdtwo{\nablaLdtwo\nufunc{\fF}-\nablaLdtwo\nufunc{\gG}}\leq\LL\normLdtwo{\fF-\gG}.
\]
In addition, a function $\nuu:\Ldtwo\to\mathbb{R}$ is called convex if, for all $\fF,\gG\in\Ldtwo$ and $\lambda\in[0,1]$, one has
\[
    \nufunc{(1-\lambda)\fF+\lambda\gG}\leq(1-\lambda)\nufunc{\fF}+\lambda\nufunc{\gG}.
\]
Moreover, a function $\nuu:\Ldtwo\to\mathbb{R}$ is called $\mu$-strongly convex if $\nablaLdtwo\nufunc{\fF}$ exists for all $\fF\in\Ldtwo$ and there exists $\mu$ such that, for all $\fF,\gG\in\Ldtwo$,
\[
    \nufunc{\gG}\geq\nufunc{\fF}+\productLdtwo{\nablaLdtwo\nufunc{\fF}}{\gG-\fF}+\frac{\mu}{2}\normLdtwo{\gG-\fF}^{2}
\]
holds. In particular, the case $\mu=0$ reduces to convexity; see Proposition~\ref{prop:convex_equivalent}.
\par The following are standard facts in convex analysis. We present them here for completeness.
\begin{prop}
    Let $\nuu:\Ldtwo\to\mathbb{R}$ be $\LL$-smooth. Then, for all $\fF,\gG$, the following inequality holds:
    \[
        \nufunc{\gG}\leq\nufunc{\fF}+\productLdtwo{\nablaLdtwo\nufunc{\fF}}{\gG-\fF}+\frac{\LL}{2}\normLdtwo{\gG-\fF}^{2}.
    \]
\end{prop}
\begin{prop}
    \label{prop:convex_equivalent}
    Let $\nuu:\Ldtwo\to\mathbb{R}$ be a function such that the gradient $\nablaLdtwo\nufunc{\fF}$ exists for all $\fF\in\Ldtwo$. Then the following statements are equivalent.
    \begin{enumerate}
        \item The function $\nuu$ is convex.
        \item $\forall\fF,\gG\in\Ldtwo,\nufunc{\gG}\geq\nufunc{\fF}+\productLdtwo{\nablaLdtwo\nufunc{\fF}}{\gG-\fF}.$
        \item $\forall\fF,\gG\in\Ldtwo,\productLdtwo{\nablaLdtwo\nufunc{\fF}-\nablaLdtwo\nufunc{\gG}}{\fF-\gG}\geq 0.$
    \end{enumerate}
\end{prop}
\begin{prop}
    \label{prop:mu_min}
    Let $\mu$ be a positive constant. Then, a $\mu$-strongly convex function $\nuu:\Ldtwo\to\mathbb{R}$ attains its minimum at a unique point $\fstar\in\Ldtwo$, and
    \[
        \nablaLdtwo\nufunc{\fstar}=0.
    \]
    Moreover, the equation $\nablaLdtwo\nufunc{\fF}=0$ holds only when $\fF=\fstar$.
\end{prop}
\subsection{Notation for Solutions}
In this paper, we denote the exact solution of an equation by $\ufunc{x}{y}{z}{t}$, and we take $\Un{n}\in\Ldtwo$ so that $\Uijkn{i}{j}{k}{n}$ approximates $\ufunc{i\Dx}{j\Dy}{k\Dz}{n\Dt}$, where $\Dt$ is the time step. We assume that the initial values coincide, i.e.,
\[
    \Uijkn{i}{j}{k}{0}=\ufunc{i\Dx}{j\Dy}{k\Dz}{0}.
\]
We denote the error $\en{n}\in\Ldtwo$ by
\[
    \eijkn{i}{j}{k}{n}\coloneqq\Uijkn{i}{j}{k}{n}-\ufunc{i\Dx}{j\Dy}{k\Dz}{n\Dt}.
\]
\par For simplicity, we denote the function $\unn{n}:\Omega\to\mathbb{R}$ as
\[
    \unnfunc{x}{y}{z}{n}\coloneqq\ufunc{x}{y}{z}{n\Dt}.
\]
In addition, let $\displaystyle\un{n},\dudtn{n},\Deltaun{n},\Deltatwoun{n}\in\Ldtwo$ be defined by
\[
    \begin{split}
        \uijkn{i}{j}{k}{n}&\coloneqq\ufunc{i\Dx}{j\Dy}{k\Dz}{n\Dt},\\
        \dudtijkn{i}{j}{k}{n}&\coloneqq\pdv{\uu}{t}\mleft(i\Dx,j\Dy,k\Dz,n\Dt\mright),\\
        \Deltauijkn{i}{j}{k}{n}&\coloneqq\Delta\uu\mleft(i\Dx,j\Dy,k\Dz,n\Dt\mright),\\
        \Deltatwouijkn{i}{j}{k}{n}&\coloneqq\Delta^{2}\uu\mleft(i\Dx,j\Dy,k\Dz,n\Dt\mright).
    \end{split}
\]
Also, for $\UU\in\Ldtwo$, we define $\UU^{3},\deltwo\UU,\delfour\UU\in\Ldtwo$ as
\[
    \begin{split}
        \left(\UU^{3}\right)_{i,j,k}&\coloneqq\Uijk{i}{j}{k}^{3},\\
        \left(\deltwo\UU\right)_{i,j,k}&\coloneqq\deltwo\Uijk{i}{j}{k},\\
        \left(\delfour\UU\right)_{i,j,k}&\coloneqq\delfour\Uijk{i}{j}{k}.
    \end{split}
\] 
    
\section{Dissipation-Preserving Scheme Framework}
\label{sec:Dissipation-Preserving_Scheme_Framework}
In this paper, we propose a linearly implicit finite difference scheme for the Swift--Hohenberg equation. The key idea is to split the right-hand side into explicit and implicit parts so as to obtain a dissipation-preserving scheme, motivated by~\cite{Iwade}. In this section, we formulate the general design principle and establish unique solvability and a discrete energy-dissipation law, together with an error estimate, under a time-step restriction and a suitable regularity assumption on the exact solution. The concrete scheme for the Swift--Hohenberg equation is presented in the next section.
\subsection{Definition of a \Iwade}
Consider the PDE
\begin{equation}
    \label{equ:exact_general}
    \pdv{\uu}{t}=-\left\{\nutilonefunc{\uu}+\nutiltwofunc{\uu}-\nutilthreefunc{\uu}\right\}
\end{equation}
on $\Omega$, together with a scheme
\[
    \frac{\Un{n+1}-\Un{n}}{\Dt}=-\left\{\nablaLdtwo\nuonefunc{\Un{n}}+\nablaLdtwo\nutwofunc{\Un{n+1}}-\nablaLdtwo\nuthreefunc{\Un{n}}\right\},
\]
where $\nuone,\nutwo,\nuthree:\Ldtwo\to\mathbb{R}$, $\tilde{\nuu}_{1},\tilde{\nuu}_{2},\tilde{\nuu}_{3}$ denote the continuous counterparts of $\nuone,\nutwo,\nuthree$, and $\nabla_{L^{2}}$ denotes the continuous counterpart of $\nablaLdtwo$ (exact definitions are omitted for brevity). We assume boundary conditions (e.g., periodic or Neumann) under which $\nabla_{L^{2}}$ and $\nablaLdtwo$ are well-defined.
\par It is not obvious how to choose a splitting of the right-hand side into $\tilde{\nuu}_{1},\tilde{\nuu}_{2},\tilde{\nuu}_{3}$. However, we show below that a splitting based on the decomposition ($\LL$-smooth) $+$ (strongly convex) $-$ (strongly convex) leads to a scheme which admits a unique solution at each time step and satisfies a dissipation law under a time-step restriction.
\begin{dfn}
    \label{dfn:iwade}
    Let $\nuone:\Ldtwo\to\mathbb{R}$ be $\LL$-smooth, and let $\nutwo,\nuthree:\Ldtwo\to\mathbb{R}$ be $\mu_{2}$- and $\mu_{3}$-strongly convex, respectively, with $\mu_{2}\geq 0$. For $\Un{n}\in\Ldtwo$, the following scheme is called a ``\Iwade'':
    \[
        \frac{\Un{n+1}-\Un{n}}{\Dt}=-\left\{\nablaLdtwo\nuonefunc{\Un{n}}+\nablaLdtwo\nutwofunc{\Un{n+1}}-\nablaLdtwo\nuthreefunc{\Un{n}}\right\}.
    \]
\end{dfn}
\begin{thm}
    \label{thm:iwade_unique}
    The scheme in Definition~\ref{dfn:iwade} admits a unique solution at each time step.
\end{thm}
\begin{proof}
    Let
    \[
        \nufournfunc{n}{\UU}\coloneqq\nutwofunc{\UU}-\productLdtwo{\UU}{-\nablaLdtwo\nuonefunc{\Un{n}}+\nablaLdtwo\nuthreefunc{\Un{n}}}+\frac{1}{2\Dt}\normLdtwo{\UU-\Un{n}}^{2}.
    \]
    It follows that $\nufourn{n}$ is $\left(\mu_{2}+\frac{1}{\Dt}\right)$-strongly convex. Therefore, as $\mu_{2}\geq 0$ in Definition~\ref{dfn:iwade}, Proposition~\ref{prop:mu_min} yields the existence of a unique $\UU$ with $\nablaLdtwo\nufournfunc{n}{\UU}=0$. Here,
    \[
        \begin{split}
            &\nablaLdtwo\nufournfunc{n}{\UU}=0\\
            \iff&\nablaLdtwo\nutwofunc{\UU}+\nablaLdtwo\nuonefunc{\Un{n}}-\nablaLdtwo\nuthreefunc{\Un{n}}+\frac{\UU-\Un{n}}{\Dt}=0,\\
        \end{split}
    \]
    and this condition is satisfied when $\UU=\Un{n+1}$. Hence, there exists a unique $\Un{n+1}$, i.e., the scheme in Definition~\ref{dfn:iwade} admits a unique solution.
\end{proof}
\begin{rem}
    \label{rem:iwade_unique}
    Theorem~\ref{thm:iwade_unique} holds without any smoothness assumption on $\nuone$. The smoothness of $\nuone$ will be used in subsequent results.
\end{rem}
\begin{thm}
    \label{thm:iwade_general_mu}
    In Definition~\ref{dfn:iwade}, let $\nusum\coloneqq\nuone+\nutwo-\nuthree$. If $\LL\leq\mu_{2}+\mu_{3}$ or $\begin{cases}
        \LL>\mu_{2}+\mu_{3}\\
        \Dt\leq\frac{2}{\LL-\mu_{2}-\mu_{3}}
    \end{cases}$ holds, $\nusumfunc{\Un{n}}$ is non-increasing in $n$.
\end{thm}
\begin{proof}
    Because $\nuone$ is $\LL$-smooth and $\nutwo,\nuthree$ are $\mu_{2}$- and $\mu_{3}$-strongly convex, respectively,
    \[
        \begin{split}
            \nuonefunc{\Un{n+1}}-\nuonefunc{\Un{n}}&\leq\productLdtwo{\nablaLdtwo\nuonefunc{\Un{n}}}{\Un{n+1}-\Un{n}}+\frac{\LL}{2}\normLdtwo{\Un{n+1}-\Un{n}}^{2},\\
            \nutwofunc{\Un{n+1}}-\nutwofunc{\Un{n}}&\leq\productLdtwo{\nablaLdtwo\nutwofunc{\Un{n+1}}}{\Un{n+1}-\Un{n}}\\
            &\quad\quad\quad\quad\quad\quad\quad\quad\quad\quad\quad\quad\quad\quad\quad\quad\quad-\frac{\mu_{2}}{2}\normLdtwo{\Un{n+1}-\Un{n}}^{2},\\
            \nuthreefunc{\Un{n+1}}-\nuthreefunc{\Un{n}}&\geq\productLdtwo{\nablaLdtwo\nuthreefunc{\Un{n}}}{\Un{n+1}-\Un{n}}+\frac{\mu_{3}}{2}\normLdtwo{\Un{n+1}-\Un{n}}^{2}.
        \end{split}
    \]
    Hence,
    \[
        \begin{split}
            &\quad\nusumfunc{\Un{n+1}}-\nusumfunc{\Un{n}}\\
            &\leq\productLdtwo{\nablaLdtwo\nuonefunc{\Un{n}}+\nablaLdtwo\nutwofunc{\Un{n+1}}-\nablaLdtwo\nuthreefunc{\Un{n}}}{\Un{n+1}-\Un{n}}\\
            &\quad+\frac{\LL-\mu_{2}-\mu_{3}}{2}\normLdtwo{\Un{n+1}-\Un{n}}^{2}\\
            &=-\left(\frac{1}{\Dt}-\frac{\LL-\mu_{2}-\mu_{3}}{2}\right)\normLdtwo{\Un{n+1}-\Un{n}}^{2}.
        \end{split}
    \]
    This is non-positive whenever $\LL\leq\mu_{2}+\mu_{3}$ or $\begin{cases}
        \LL>\mu_{2}+\mu_{3}\\
        \Dt\leq\frac{2}{\LL-\mu_{2}-\mu_{3}}
    \end{cases}$ holds.
\end{proof}
\begin{rem}
    In fact, for the Cahn--Hilliard equation, the \Iwade\ was already suggested in~\cite{CH_collision_1}, and \cite{CH_collision_2} proved that this scheme for the Cahn--Hilliard equation satisfies the energy-dissipation law.
\end{rem}

\subsection{Error Analysis for the General \Iwade}
We next give an error analysis for the general \Iwade.
\begin{thm}
    \label{thm:error_iwade}
    Assume that
    \begin{enumerate}
        \item $\nablaLdtwo\nutwo,\nablaLdtwo\nuthree,\nutiltwo$ are linear;
        \item $\nusum$ is bounded from below;
        \item the solution $\uu$ to \eqref{equ:exact_general} satisfies $\uu\in C^{2}$;
        \item $\nutiltwofunc{\pdv{\uu}{t}}$ is well-defined, and that $\Lhat\coloneqq\LL-\mu_{2}-\mu_{3}\neq 0$.
    \end{enumerate}
    Then, for any $\rr\in(0,1)$ and any positive constants $\kappan{1},\kappan{2}$ with $\kappan{1}+\kappan{2}<2$, if $\Lhat<0$ or $\begin{cases}
        \Lhat>0\\
        \Dt\leq\frac{\rr\kappan{2}}{\Lhat}
    \end{cases}$ holds, we obtain
    \[
        \begin{split}
            \normLdtwo{\en{n}}&\leq\sqrt{\frac{\kappan{1}\kappan{2}}{\left(2-\kappan{1}-\kappan{2}\right)\left(\kappan{1}\Lhat^{2}+4\kappan{2}\LL^{2}\right)}}\left(\CT\Dt+\max_{\nd\in\{0,\dots,n-1\}}\normLdtwo{\gamman{\nd+\frac{1}{2}}}\right)\\
            &\quad\quad\quad\quad\quad\quad\quad\quad\quad\quad\quad\quad\times\sqrt{\exp(\left(\frac{4\LL^{2}}{\kappan{1}\abs{\Lhat}}+\frac{\abs{\Lhat}}{\kappan{2}\left(1-\rr\thetafunc{\Lhat}\right)}\right)T)-1}
        \end{split}
    \]
    for a constant $\CT$ depending only on $T\coloneqq n\Dt$, $\nutiltwo$, and the exact solution. Here, $\thetafunc{\tilde{\eta}}$ is the step function
    \[
        \thetafunc{\tilde{\eta}}\coloneqq\begin{dcases}
            1 & (\tilde{\eta}>0),\\
            0 & (\tilde{\eta}\leq 0),
        \end{dcases}
    \]
    and
    \[
        \begin{split}
            \gamman{n+\frac{1}{2}}&\coloneqq\left\{\nablaLdtwo\nuonefunc{\un{n}}+\nablaLdtwo\nutwofunc{\un{n+1}}-\nablaLdtwo\nuthreefunc{\un{n}}\right\}\\
            &\quad-\left\{\iotaafunc{\nutilonefunc{\unn{n}}}+\iotaafunc{\nutiltwofunc{\unn{n+1}}}-\iotaafunc{\nutilthreefunc{\unn{n}}}\right\}.
        \end{split}
    \]
\end{thm}
\par Its proof will be given after Remark~\ref{rem:not-essential}.
\begin{rem}
    In Theorem~\ref{thm:error_iwade}, the parameters $\kappan{1}$ and $\kappan{2}$ arise from applications of Young's inequality in the proof, while the parameter $\rr$ is introduced in order to absorb an undesirable $\Dt$-dependent term from the error bound. Since $\rr$ and $\kappan{2}$ affect both the time-step restriction and the constants in the bound, it is not straightforward to identify the optimal choice of $\rr,\kappan{1},\kappan{2}$. For the reader's convenience, we also state a simplified version with fixed parameters.
\end{rem}
\begin{cor}
    \label{cor:error_iwade_special}
    Assume that the four assumptions in Theorem~\ref{thm:error_iwade} hold. Then, if $\Lhat<0$ or $\begin{cases}
        \Lhat>0\\
        \Dt\leq\frac{1}{3\Lhat}
    \end{cases}$ holds, for a constant $\CT$ depending only on $T$, $\nutiltwo$, and the exact solution, and a constant $\CTd$ depending only on $\LL,\Lhat,T$,
    \[
        \normLdtwo{\en{n}}\leq\CTd\left(\CT\Dt+\max_{\nd\in\{0,\dots,n-1\}}\normLdtwo{\gamman{\nd+\frac{1}{2}}}\right).
    \]
\end{cor}
\begin{proof}
    It follows straightforwardly by setting $\rr=\frac{1}{2}$ and $\kappan{1}=\kappan{2}=\frac{2}{3}$ in Theorem~\ref{thm:error_iwade}. Note that the choice is made for simplicity and is not intended to be optimal.
\end{proof}
\begin{rem}
    In Theorem~\ref{thm:error_iwade} and Corollary~\ref{cor:error_iwade_special}, we bound $\normLdtwo{\en{n}}$ from above. The term $\CT\Dt$ corresponds to the time-discretization error, while $\max_{\nd\in\{0,\dots,n-1\}}\normLdtwo{\gamman{\nd+\frac{1}{2}}}$ corresponds to the spatial-discretization error. Hence, the time-discretization error vanishes as $\Dt\to 0$.
\end{rem}
\begin{rem}
    \label{rem:not-essential}
    We assume that $\Lhat\neq 0$ in Theorem~\ref{thm:error_iwade} and Corollary~\ref{cor:error_iwade_special}. This is not essential, since an $\LL$-smooth function is also $\LL'$-smooth for any $\LL'\geq\LL$.
\end{rem}
\begin{proof}[Proof of Theorem~\ref{thm:error_iwade}]
    First observe that
    \[
        \begin{split}
            \frac{\en{n+1}-\en{n}}{\Dt}&=\left(\frac{\Un{n+1}-\Un{n}}{\Dt}-\dudtn{n}\right)-\left(\frac{\un{n+1}-\un{n}}{\Dt}-\dudtn{n}\right)\\
            &=-\left(\left\{\nablaLdtwo\nuonefunc{\Un{n}}+\nablaLdtwo\nutwofunc{\Un{n+1}}-\nablaLdtwo\nuthreefunc{\Un{n}}\right\}\right.\\
            &\quad\quad\quad\left.-\left\{\nablaLdtwo\nuonefunc{\un{n}}+\nablaLdtwo\nutwofunc{\un{n+1}}-\nablaLdtwo\nuthreefunc{\un{n}}\right\}\right)\\
            &\quad-\left(\left\{\nablaLdtwo\nuonefunc{\un{n}}+\nablaLdtwo\nutwofunc{\un{n+1}}-\nablaLdtwo\nuthreefunc{\un{n}}\right\}\right.\\
            &\quad\quad\quad\left.-\left\{\iotaafunc{\nutilonefunc{\unn{n}}}+\iotaafunc{\nutiltwofunc{\unn{n}}}-\iotaafunc{\nutilthreefunc{\unn{n}}}\right\}\right)\\
            &\quad-\left(\frac{\un{n+1}-\un{n}}{\Dt}-\dudtn{n}\right)\\
            &=-\left\{\nablaLdtwo\nuonefunc{\Un{n}}-\nablaLdtwo\nuonefunc{\un{n}}+\nablaLdtwo\nutwofunc{\en{n+1}}-\nablaLdtwo\nuthreefunc{\en{n}}\right\}\\
            &\quad-\gamman{n+\frac{1}{2}}-\gammatiln{n+\frac{1}{2}},
        \end{split}
    \]
    where
    \[
        \gammatiln{n+\frac{1}{2}}\coloneqq\iotaafunc{\nutiltwofunc{\unn{n+1}-\unn{n}}}+\left(\frac{\un{n+1}-\un{n}}{\Dt}-\dudtn{n}\right).
    \]
    Define
    \[
        \chimfunc{\mm}{\UU}\coloneqq\nusumfunc{\Um{\mm}+\UU}-\inf_{\hat{\UU}\in\Ldtwo}\nusumfunc{\hat{\UU}}+\frac{\abs{\Lhat}}{2}\normLdtwo{\UU}^{2},
    \]
    where $\left\{\Um{\mm}\right\}$ is a sequence in $\Ldtwo$ such that
    \[
        \lim_{\mm\to\infty}\normLdtwo{\nablaLdtwo\nusumfunc{\Um{\mm}}}=0.
    \]
    Such a sequence exists by Corollary~4.1 in~\cite{bounded_gradient_sequence}. Since $\nuone,\nutwo,\nuthree$ are $\LL$-smooth, $\mu_{2}$-strongly convex, and $\mu_{3}$-strongly convex, respectively, we obtain
    \[
        \begin{split}
            &\quad\chimfunc{\mm}{\en{n+1}}-\chimfunc{\mm}{\en{n}}\\
            &\leq\productLdtwo{\nablaLdtwo\nuonefunc{\Um{\mm}+\en{n}}+\nablaLdtwo\nutwofunc{\Um{\mm}+\en{n+1}}\right.\\
            &\quad\quad\left.-\nablaLdtwo\nuthreefunc{\Um{\mm}+\en{n}}+\frac{\abs{\Lhat}}{2}\left(\en{n+1}+\en{n}\right)}{\en{n+1}-\en{n}}\\
            &\quad+\frac{\Lhat}{2}\normLdtwo{\en{n+1}-\en{n}}^{2}\\
            &=\productLdtwo{\nablaLdtwo\nuonefunc{\Um{\mm}+\en{n}}+\left(\nablaLdtwo\nutwofunc{\Um{\mm}}-\nablaLdtwo\nuthreefunc{\Um{\mm}}\right)\right.\\
            &\quad\quad\left.+\left(\nablaLdtwo\nutwofunc{\en{n+1}}-\nablaLdtwo\nuthreefunc{\en{n}}\right)-(-1)^{\thetafunc{\Lhat}}\Lhat\en{n+\thetafunc{\Lhat}}}{\en{n+1}-\en{n}}\\
            &=\productLdtwo{\nablaLdtwo\nuonefunc{\Um{\mm}+\en{n}}+\nablaLdtwo\nusumfunc{\Um{\mm}}-\nablaLdtwo\nuonefunc{\Um{\mm}}\right.\\
            &\quad\quad\quad\quad\quad\quad\quad-\frac{\en{n+1}-\en{n}}{\Dt}-\nablaLdtwo\nuonefunc{\Un{n}}+\nablaLdtwo\nuonefunc{\un{n}}\\
            &\quad\quad\quad\quad\quad\quad\quad\quad\quad\quad\left.-\gamman{n+\frac{1}{2}}-\gammatiln{n+\frac{1}{2}}-(-1)^{\thetafunc{\Lhat}}\Lhat\en{n+\thetafunc{\Lhat}}}{\en{n+1}-\en{n}}\\
            &\leq\left(\normLdtwo{\nablaLdtwo\nuonefunc{\Um{\mm}+\en{n}}-\nablaLdtwo\nuonefunc{\Um{\mm}}}+\normLdtwo{\nablaLdtwo\nuonefunc{\Un{n}}-\nablaLdtwo\nuonefunc{\un{n}}}\right.\\
            &\quad\quad\left.+\normLdtwo{\gamman{n+\frac{1}{2}}+\gammatiln{n+\frac{1}{2}}-\nablaLdtwo\nusumfunc{\Um{\mm}}}+\abs{\Lhat}\normLdtwo{\en{n+\thetafunc{\Lhat}}}\right)\normLdtwo{\en{n+1}-\en{n}}\\
            &\quad-\frac{1}{\Dt}\normLdtwo{\en{n+1}-\en{n}}^{2}\\
            &\leq\left(2\LL\normLdtwo{\en{n}}+\normLdtwo{\gamman{n+\frac{1}{2}}+\gammatiln{n+\frac{1}{2}}-\nablaLdtwo\nusumfunc{\Um{\mm}}}\right.\\
            &\quad\quad\quad\quad\quad\quad\quad\quad\quad\quad\quad\quad\quad\quad\quad\quad\quad\quad\quad\quad\left.+\abs{\Lhat}\normLdtwo{\en{n+\thetafunc{\Lhat}}}\right)\normLdtwo{\en{n+1}-\en{n}}\\
            &\quad-\frac{1}{\Dt}\normLdtwo{\en{n+1}-\en{n}}^{2}\\
            &\leq\left(\frac{2\LL^{2}\Dt}{\kappan{1}}\normLdtwo{\en{n}}^{2}+\frac{\kappan{1}}{2\Dt}\normLdtwo{\en{n+1}-\en{n}}^{2}\right)\\
            &\quad+\left(\frac{\Dt}{2\left(2-\kappan{1}-\kappan{2}\right)}\normLdtwo{\gamman{n+\frac{1}{2}}+\gammatiln{n+\frac{1}{2}}-\nablaLdtwo\nusumfunc{\Um{\mm}}}^{2}\right.\\
            &\quad\quad\quad\quad\quad\quad\quad\quad\quad\quad\quad\quad\quad\quad\quad\quad\quad\quad\quad\quad\left.+\frac{2-\kappan{1}-\kappan{2}}{2\Dt}\normLdtwo{\en{n+1}-\en{n}}^{2}\right)\\
            &\quad+\left(\frac{\Lhat^{2}\Dt}{2\kappan{2}}\normLdtwo{\en{n+\thetafunc{\Lhat}}}^{2}+\frac{\kappan{2}}{2\Dt}\normLdtwo{\en{n+1}-\en{n}}^{2}\right)-\frac{1}{\Dt}\normLdtwo{\en{n+1}-\en{n}}^{2}\\
        \end{split}
    \]
    \[
        \begin{split}
            &\leq\frac{4\LL^{2}\Dt}{\kappan{1}\abs{\Lhat}}\chimfunc{\mm}{\en{n}}+\frac{\abs{\Lhat}\Dt}{\kappan{2}}\chimfunc{\mm}{\en{n+\thetafunc{\Lhat}}}\\
            &\quad+\frac{\Dt}{2\left(2-\kappan{1}-\kappan{2}\right)}\normLdtwo{\gamman{n+\frac{1}{2}}+\gammatiln{n+\frac{1}{2}}-\nablaLdtwo\nusumfunc{\Um{\mm}}}^{2}.
        \end{split}
    \]
    The last inequality follows from $\chimfunc{\mm}{\UU}\geq\frac{\abs{\Lhat}}{2}\normLdtwo{\UU}^{2}$. Hence, since $\Dt\leq\frac{\rr\kappan{2}}{\Lhat}<\frac{\kappan{2}}{\Lhat}$ holds whenever $\Lhat>0$, 
    \[
        \begin{split}
            &\quad\chimfunc{\mm}{\en{n+1}}\\
            \leq&\begin{dcases}
                \frac{\kappan{2}\left(\kappan{1}\Lhat+4\LL^{2}\Dt\right)}{\kappan{1}\Lhat\left(\kappan{2}-\Lhat\Dt\right)}\chimfunc{\mm}{\en{n}}\\
                \quad+\frac{\kappan{2}\Dt}{2\left(2-\kappan{1}-\kappan{2}\right)\left(\kappan{2}-\Lhat\Dt\right)}\normLdtwo{\gamman{n+\frac{1}{2}}+\gammatiln{n+\frac{1}{2}}-\nablaLdtwo\nusumfunc{\Um{\mm}}}^{2} & (\Lhat>0),\\
                \left(1+\frac{4\LL^{2}\Dt}{\kappan{1}\abs{\Lhat}}+\frac{\abs{\Lhat}\Dt}{\kappan{2}}\right)\chimfunc{\mm}{\en{n}}\\
                \quad+\frac{\Dt}{2\left(2-\kappan{1}-\kappan{2}\right)}\normLdtwo{\gamman{n+\frac{1}{2}}+\gammatiln{n+\frac{1}{2}}-\nablaLdtwo\nusumfunc{\Um{\mm}}}^{2} & (\Lhat<0).\\
            \end{dcases}
        \end{split}
    \]
    Therefore,
    \[
        \begin{split}
            &\quad\chimfunc{\mm}{\en{n}}\\
            &\leq\begin{dcases}
                \left(\frac{\kappan{2}\left(\kappan{1}\Lhat+4\LL^{2}\Dt\right)}{\kappan{1}\Lhat\left(\kappan{2}-\Lhat\Dt\right)}\right)^{n}\chimfunc{\mm}{\en{0}}\\
                \quad+\sum_{\nd=0}^{n-1}\frac{\kappan{2}\Dt}{2\left(2-\kappan{1}-\kappan{2}\right)\left(\kappan{2}-\Lhat\Dt\right)}\left(\frac{\kappan{2}\left(\kappan{1}\Lhat+4\LL^{2}\Dt\right)}{\kappan{1}\Lhat\left(\kappan{2}-\Lhat\Dt\right)}\right)^{\nd}\\
                \quad\quad\quad\quad\times\normLdtwo{\gamman{n-\nd-\frac{1}{2}}+\gammatiln{n-\nd-\frac{1}{2}}-\nablaLdtwo\nusumfunc{\Um{\mm}}}^{2} & (\Lhat>0)\\
                \left(1+\frac{4\LL^{2}\Dt}{\kappan{1}\abs{\Lhat}}+\frac{\abs{\Lhat}\Dt}{\kappan{2}}\right)^{n}\chimfunc{\mm}{\en{0}}\\
                \quad+\sum_{\nd=0}^{n-1}\frac{\Dt}{2\left(2-\kappan{1}-\kappan{2}\right)}\left(1+\frac{4\LL^{2}\Dt}{\kappan{1}\abs{\Lhat}}+\frac{\abs{\Lhat}\Dt}{\kappan{2}}\right)^{\nd}\\
                \quad\quad\quad\quad\times\normLdtwo{\gamman{n-\nd-\frac{1}{2}}+\gammatiln{n-\nd-\frac{1}{2}}-\nablaLdtwo\nusumfunc{\Um{\mm}}}^{2} & (\Lhat<0)
            \end{dcases}\\
        \end{split}
    \]
    \[
        \begin{split}
            &=\begin{dcases}
                \left(\frac{\kappan{2}\left(\kappan{1}\Lhat+4\LL^{2}\Dt\right)}{\kappan{1}\Lhat\left(\kappan{2}-\Lhat\Dt\right)}\right)^{n}\chimfunc{\mm}{\en{0}}\\
                \quad+\frac{\displaystyle\kappan{2}\Dt\cdot\max_{\nd\in\{0,\dots,n-1\}}\normLdtwo{\gamman{\nd+\frac{1}{2}}+\gammatiln{\nd+\frac{1}{2}}-\nablaLdtwo\nusumfunc{\Um{\mm}}}^{2}}{2\left(2-\kappan{1}-\kappan{2}\right)\left(\kappan{2}-\Lhat\Dt\right)}\\
                \quad\quad\quad\quad\quad\quad\quad\quad\quad\quad\quad\quad\quad\quad\quad\times\frac{1-\left(\frac{\kappan{2}\left(\kappan{1}\Lhat+4\LL^{2}\Dt\right)}{\kappan{1}\Lhat\left(\kappan{2}-\Lhat\Dt\right)}\right)^{n}}{1-\frac{\kappan{2}\left(\kappan{1}\Lhat+4\LL^{2}\Dt\right)}{\kappan{1}\Lhat\left(\kappan{2}-\Lhat\Dt\right)}} & (\Lhat>0)\\
                \left(1+\frac{4\LL^{2}\Dt}{\kappan{1}\abs{\Lhat}}+\frac{\abs{\Lhat}\Dt}{\kappan{2}}\right)^{n}\chimfunc{\mm}{\en{0}}\\
                \quad+\frac{\displaystyle\Dt\cdot\max_{\nd\in\{0,\dots,n-1\}}\normLdtwo{\gamman{\nd+\frac{1}{2}}+\gammatiln{\nd+\frac{1}{2}}-\nablaLdtwo\nusumfunc{\Um{\mm}}}^{2}}{2\left(2-\kappan{1}-\kappan{2}\right)}\\
                \quad\quad\quad\quad\quad\quad\quad\quad\quad\quad\quad\quad\quad\quad\quad\times\frac{1-\left(1+\frac{4\LL^{2}\Dt}{\kappan{1}\abs{\Lhat}}+\frac{\abs{\Lhat}\Dt}{\kappan{2}}\right)^{n}}{1-\left(1+\frac{4\LL^{2}\Dt}{\kappan{1}\abs{\Lhat}}+\frac{\abs{\Lhat}\Dt}{\kappan{2}}\right)} & (\Lhat<0)
            \end{dcases}\\
            &\leq\begin{dcases}
                \left(\frac{\kappan{2}\left(\kappan{1}\Lhat+4\LL^{2}\Dt\right)}{\kappan{1}\Lhat\left(\kappan{2}-\Lhat\Dt\right)}\right)^{n}\chimfunc{\mm}{\en{0}}\\
                \quad+\frac{\displaystyle\kappan{1}\kappan{2}\Lhat\max_{\nd\in\{0,\dots,n-1\}}\normLdtwo{\gamman{\nd+\frac{1}{2}}+\gammatiln{\nd+\frac{1}{2}}-\nablaLdtwo\nusumfunc{\Um{\mm}}}^{2}}{2\left(2-\kappan{1}-\kappan{2}\right)\left(\kappan{1}\Lhat^{2}+4\kappan{2}\LL^{2}\right)}\\
                \quad\quad\quad\quad\quad\quad\quad\quad\quad\quad\quad\quad\quad\quad\quad\times\left\{\left(\frac{\kappan{2}\left(\kappan{1}\Lhat+4\LL^{2}\Dt\right)}{\kappan{1}\Lhat\left(\kappan{2}-\Lhat\Dt\right)}\right)^{n}-1\right\} & (\Lhat>0),\\
                \left(1+\frac{4\LL^{2}\Dt}{\kappan{1}\abs{\Lhat}}+\frac{\abs{\Lhat}\Dt}{\kappan{2}}\right)^{n}\chimfunc{\mm}{\en{0}}\\
                \quad+\frac{\displaystyle\kappan{1}\kappan{2}\abs{\Lhat}\max_{\nd\in\{0,\dots,n-1\}}\normLdtwo{\gamman{\nd+\frac{1}{2}}+\gammatiln{\nd+\frac{1}{2}}-\nablaLdtwo\nusumfunc{\Um{\mm}}}^{2}}{2\left(2-\kappan{1}-\kappan{2}\right)\left(\kappan{1}\Lhat^{2}+4\kappan{2}\LL^{2}\right)}\\
                \quad\quad\quad\quad\quad\quad\quad\quad\quad\quad\quad\quad\quad\quad\quad\times\left\{\left(1+\frac{4\LL^{2}\Dt}{\kappan{1}\abs{\Lhat}}+\frac{\abs{\Lhat}\Dt}{\kappan{2}}\right)^{n}-1\right\} & (\Lhat<0).
            \end{dcases}
        \end{split}
    \]
    Here, 
    \[
        \begin{split}
            \frac{\kappan{2}\left(\kappan{1}\Lhat+4\LL^{2}\Dt\right)}{\kappan{1}\Lhat\left(\kappan{2}-\Lhat\Dt\right)}&=\left(1+\frac{4\LL^{2}\Dt}{\kappan{1}\Lhat}\right)\left(1+\frac{\frac{\Lhat\Dt}{\kappan{2}}}{1-\frac{\Lhat\Dt}{\kappan{2}}}\right)\\
            &\leq\exp(\left(\frac{4\LL^{2}}{\kappan{1}\Lhat}+\frac{\Lhat}{\kappan{2}-\Lhat\Dt}\right)\Dt),
        \end{split}
    \]
    and
    \[
        1+\frac{4\LL^{2}\Dt}{\kappan{1}\abs{\Lhat}}+\frac{\abs{\Lhat}\Dt}{\kappan{2}}\leq\exp(\left(\frac{4\LL^{2}}{\kappan{1}\abs{\Lhat}}+\frac{\abs{\Lhat}}{\kappan{2}}\right)\Dt)
    \]
    hold. Combining these relationships with $\chimfunc{\mm}{\UU}\geq\frac{\abs{\Lhat}}{2}\normLdtwo{\UU}^{2}$, using $\en{0}=0$, and taking the limit $\mm\to\infty$, we derive
    \[
        \begin{split}
            &\quad\normLdtwo{\en{n}}^{2}\\
            &\leq\begin{dcases}
                \frac{\kappan{1}\kappan{2}}{\left(2-\kappan{1}-\kappan{2}\right)\left(\kappan{1}\Lhat^{2}+4\kappan{2}\LL^{2}\right)}\max_{\nd\in\{0,\dots,n-1\}}\normLdtwo{\gamman{\nd+\frac{1}{2}}+\gammatiln{\nd+\frac{1}{2}}}^{2}\\
                \quad\quad\quad\quad\quad\quad\quad\quad\quad\quad\quad\quad\times\left\{\exp(\left(\frac{4\LL^{2}}{\kappan{1}\Lhat}+\frac{\Lhat}{\kappan{2}-\Lhat\Dt}\right)T)-1\right\} & (\Lhat>0),\\
                \frac{\kappan{1}\kappan{2}}{\left(2-\kappan{1}-\kappan{2}\right)\left(\kappan{1}\Lhat^{2}+4\kappan{2}\LL^{2}\right)}\max_{\nd\in\{0,\dots,n-1\}}\normLdtwo{\gamman{\nd+\frac{1}{2}}+\gammatiln{\nd+\frac{1}{2}}}^{2}\\
                \quad\quad\quad\quad\quad\quad\quad\quad\quad\quad\quad\quad\times\left\{\exp(\left(\frac{4\LL^{2}}{\kappan{1}\abs{\Lhat}}+\frac{\abs{\Lhat}}{\kappan{2}}\right)T)-1\right\} & (\Lhat<0),
            \end{dcases}
        \end{split}
    \]
    namely
    \[
        \begin{split}
            \normLdtwo{\en{n}}&\leq\sqrt{\frac{\kappan{1}\kappan{2}}{\left(2-\kappan{1}-\kappan{2}\right)\left(\kappan{1}\Lhat^{2}+4\kappan{2}\LL^{2}\right)}}\max_{\nd\in\{0,\dots,n-1\}}\normLdtwo{\gamman{\nd+\frac{1}{2}}+\gammatiln{\nd+\frac{1}{2}}}\\
            &\quad\times\sqrt{\exp(\left(\frac{4\LL^{2}}{\kappan{1}\abs{\Lhat}}+\frac{\abs{\Lhat}}{\kappan{2}\left(1-\rr\thetafunc{\Lhat}\right)}\right)T)-1},
        \end{split}
    \]
    since $\Dt\leq\frac{\rr\kappan{2}}{\Lhat}$ whenever $\Lhat>0$. Finally, $\gammatiln{\nd+\frac{1}{2}}$ can be expressed as a constant multiple of $\Dt$ by Taylor's theorem, with coefficient depending only on $\nutiltwo$, $\uu$, and $\pdv{\uu}{t}$. Hence,
    \[
        \begin{split}
            \normLdtwo{\en{n}}&\leq\sqrt{\frac{\kappan{1}\kappan{2}}{\left(2-\kappan{1}-\kappan{2}\right)\left(\kappan{1}\Lhat^{2}+4\kappan{2}\LL^{2}\right)}}\max_{\nd\in\{0,\dots,n-1\}}\left(\normLdtwo{\gamman{\nd+\frac{1}{2}}}+\CT\Dt\right)\\
            &\quad\times\sqrt{\exp(\left(\frac{4\LL^{2}}{\kappan{1}\abs{\Lhat}}+\frac{\abs{\Lhat}}{\kappan{2}\left(1-\rr\thetafunc{\Lhat}\right)}\right)T)-1}.
        \end{split}
    \]
\end{proof}
\section{A \Iwade \ for the Swift--Hohenberg Equation}
\label{sec:A_Dissipation-Preserving_Scheme_for_the_Swift--Hohenberg_Equation}
In this section, we propose a new numerical scheme for the Swift--Hohenberg equation; we prove unique solvability at each time step, establish an energy-dissipation law and uniform bounds, and derive an error estimate, building on Section~\ref{sec:Dissipation-Preserving_Scheme_Framework}. We also discuss the scheme's asymptotic behavior.
\subsection{Properties of the Swift--Hohenberg equation}
\label{subsec:Properties_of_the_Swift--Hohenberg_equation}
We begin by briefly recalling basic properties of the Swift--Hohenberg equation \eqref{equ:SH}. It satisfies an energy-dissipation law, which follows from a straightforward calculation:
\[
    \dv{\Htruefunc{\uu}}{t}=-\int_{\Omega}\left(\pdv{\uu}{t}\right)^{2}\dd{\bm{x}}\leq 0.
\]
Moreover, classical solutions are uniformly bounded in time; specifically, for all $t\geq 0$,
\begin{equation}
    \label{equ:u_bound}
    \max_{(x,y,z)\in\Omega}\left|u\right|\leq\sqrt{\frac{\CCd^{2}}{\Ceez}\left(\Htruefunc{\uu_{0}}+\frac{\zetaa^{2}}{4}|\Omega|\right)},
\end{equation}
where $\uu_{0}$ is the initial condition, $\zetaa>\frac{1}{\varepsilon}+\eta-1$ is an arbitrary constant, $\CCd$ is a constant, and $\Ceez>0$ is defined by
\[
    \Ceez\coloneqq\frac{\varepsilon+1-\eta+\zetaa-\sqrt{(\varepsilon-1+\eta-\zetaa)^{2}+4}}{6}.
\]
The proof is given in the Appendix.
\subsection{Proposed Numerical Scheme}
We now introduce the \Iwade \ for the Swift--Hohenberg equation.
\begin{dfn}
    \label{dfn:iwade_SH}
    Define the scheme for the Swift--Hohenberg equation as follows:
    \[
        \frac{\Un{n+1}-\Un{n}}{\Dt}=\begin{dcases}
            -\left[\Un{n}^{3}+\left\{(1-\eta)\Un{n+1}+\varepsilon\delfour\Un{n+1}\right\}+2\deltwo\Un{n}\right] & (\eta<1),\\
            -\left[\Un{n}^{3}+\varepsilon\delfour\Un{n+1}-\left\{(\eta-1)\Un{n}-2\deltwo\Un{n}\right\}\right] & (\eta\geq 1),
        \end{dcases}
    \]
    with the following boundary conditions:
    \[
        \begin{dcases}
            \deliU{0}=\deliU{\Nx}=0,\\
            \deljU{0}=\deljU{\Ny}=0,\\
            \delkU{0}=\delkU{\Nz}=0,\\
            \delidelijkU{0}=\delidelijkU{\Nx}=0,\\
            \deljdelijkU{0}=\deljdelijkU{\Ny}=0,\\
            \delkdelijkU{0}=\delkdelijkU{\Nz}=0.\\
        \end{dcases}
    \]
\end{dfn}
\begin{cor}
    \label{cor:dfn_phi}
    Let
    \[
        \begin{dcases}
            \phionefunc{\UU}\coloneqq\sumd{i=0}{\Nx}\sumd{j=0}{\Ny}\sumd{k=0}{\Nz}\frac{1}{4}\Uijk{i}{j}{k}^{4}\Dx\Dy\Dz,\\
            \phitwofunc{\UU}\coloneqq\sumd{i=0}{\Nx}\sumd{j=0}{\Ny}\sumd{k=0}{\Nz}\frac{1-\eta}{2}\Uijk{i}{j}{k}^{2}\Dx\Dy\Dz,\\
            \phithreefunc{\UU}\coloneqq\sumd{i=0}{\Nx}\sumd{j=0}{\Ny}\sumd{k=0}{\Nz}\frac{\left|\nabladp\Uijk{i}{j}{k}\right|^{2}+\left|\nabladm\Uijk{i}{j}{k}\right|^{2}}{2}\Dx\Dy\Dz,\\
            \phifourfunc{\UU}\coloneqq\sumd{i=0}{\Nx}\sumd{j=0}{\Ny}\sumd{k=0}{\Nz}\frac{\varepsilon}{2}\left(\deltwo\Uijk{i}{j}{k}\right)^{2}\Dx\Dy\Dz.\\
        \end{dcases}
    \]
    Then, the scheme in Definition~\ref{dfn:iwade_SH} can be rewritten as follows:
    \[
        \begin{split}
            &\frac{\Un{n+1}-\Un{n}}{\Dt}\\
            &\quad=\begin{dcases}
                -\left[\nablaLdtwo\phionefunc{\Un{n}}+\left\{\nablaLdtwo\phitwofunc{\Un{n+1}}+\nablaLdtwo\phifourfunc{\Un{n+1}}\right\}-\nablaLdtwo\phithreefunc{\Un{n}}\right]\\
                \quad\quad\quad\quad\quad\quad\quad\quad\quad\quad\quad\quad\quad\quad\quad\quad\quad\quad\quad\quad\quad\quad\quad\quad\quad\quad\quad\quad\quad\quad(\eta<1),\\
                -\left[\nablaLdtwo\phionefunc{\Un{n}}+\nablaLdtwo\phifourfunc{\Un{n+1}}-\left\{-\nablaLdtwo\phitwofunc{\Un{n}}+\nablaLdtwo\phithreefunc{\Un{n}}\right\}\right]\\
                \quad\quad\quad\quad\quad\quad\quad\quad\quad\quad\quad\quad\quad\quad\quad\quad\quad\quad\quad\quad\quad\quad\quad\quad\quad\quad\quad\quad\quad\quad(\eta\geq 1).
            \end{dcases}
        \end{split}
    \]
\end{cor}
\begin{proof}
    \[
        \begin{dcases}
            \lim_{\Lambdaa\to +0}\frac{\phionefunc{\UU+\Lambdaa\xii}-\phionefunc{\UU}}{\Lambdaa}=\productLdtwo{\UU^{3}}{\xii},\\
            \lim_{\Lambdaa\to +0}\frac{\phitwofunc{\UU+\Lambdaa\xii}-\phitwofunc{\UU}}{\Lambdaa}=\productLdtwo{(1-\eta)\UU}{\xii},\\
            \lim_{\Lambdaa\to +0}\frac{\phithreefunc{\UU+\Lambdaa\xii}-\phithreefunc{\UU}}{\Lambdaa}=\productLdtwo{-2\deltwo\UU}{\xii},\\
            \lim_{\Lambdaa\to +0}\frac{\phifourfunc{\UU+\Lambdaa\xii}-\phifourfunc{\UU}}{\Lambdaa}=\productLdtwo{\varepsilon\delfour\UU}{\xii},
        \end{dcases}
    \]
    which follow from Lemma~\ref{lem:integration_by_parts} and the boundary conditions. Hence,
    \[
        \begin{dcases}
            \nablaLdtwo\phionefunc{\UU}=\UU^{3},\\
            \nablaLdtwo\phitwofunc{\UU}=(1-\eta)\UU,\\
            \nablaLdtwo\phithreefunc{\UU}=-2\deltwo\UU,\\
            \nablaLdtwo\phifourfunc{\UU}=\varepsilon\delfour\UU.\\
        \end{dcases}
    \]
    This completes the proof.
\end{proof}
\subsection{Stability and Dissipation Properties}
We next establish stability and dissipation properties of the scheme in Definition~\ref{dfn:iwade_SH}. Our approach follows ideas presented in~\cite{Iwade_memo_2_senko_notSH_dissipative_1,Iwade_memo_1}. 
\par We first verify the smoothness and convexity of the terms in Definition~\ref{dfn:iwade_SH}. These properties are used to invoke Theorem~\ref{thm:iwade_unique} for unique solvability and to prepare for an application of Theorem~\ref{thm:iwade_general_mu} to derive an energy-dissipation law. Proving smoothness of $\phione$ directly is difficult; instead, in Lemma~\ref{lem:SH_L_convex_mu}, we introduce a modified functional $\phionetilde$ and establish its smoothness. We justify this replacement a posteriori via a uniform $\normLd{\infty}{\placeholder}$ bound.
\begin{lem}
    \label{lem:SH_L_convex_mu}
    Let $M\geq 0$, and define the function
    \[
        \psifunc{\Uijk{i}{j}{k}}\coloneqq\begin{dcases}
            \frac{3M^{2}}{2}\Uijk{i}{j}{k}^{2}+2M^{3}\Uijk{i}{j}{k}+\frac{3}{4}M^{4} & (\Uijk{i}{j}{k}<-M),\\
            \frac{1}{4}\Uijk{i}{j}{k}^{4} & (-M\leq\Uijk{i}{j}{k}\leq M),\\
            \frac{3M^{2}}{2}\Uijk{i}{j}{k}^{2}-2M^{3}\Uijk{i}{j}{k}+\frac{3}{4}M^{4} & (\Uijk{i}{j}{k}>M).
        \end{dcases}
    \]
    Using this, define
    \[
        \phionetildefunc{\UU}\coloneqq\sumd{i=0}{\Nx}\sumd{j=0}{\Ny}\sumd{k=0}{\Nz}\psifunc{\Uijk{i}{j}{k}}\Dx\Dy\Dz.
    \]
    Then, $\phionetilde$ is $3M^{2}$-smooth. Moreover, if $\eta<1$, then $\phitwo+\phifour$ is $(1-\eta)$-strongly convex and $\phithree$ is convex, and if $\eta\geq 1$, then $\phifour$ is convex and $-\phitwo+\phithree$ is $(\eta-1)$-strongly convex.
\end{lem}
\begin{proof}
    We first determine $\nablaLdtwo\phionetilde$. We have
    \[
        \lim_{\Lambdaa\to +0}\frac{\phionetildefunc{\UU+\Lambdaa\xii}-\phionetildefunc{\UU}}{\Lambdaa}=\sumd{i=0}{\Nx}\sumd{j=0}{\Ny}\sumd{k=0}{\Nz}\lim_{\Lambdaa\to +0}\frac{\psifunc{\Uijk{i}{j}{k}+\Lambdaa\xiijk{i}{j}{k}}-\psifunc{\Uijk{i}{j}{k}}}{\Lambdaa}\Dx\Dy\Dz.
    \]
    We rewrite this as
    \[
        \lim_{\Lambdaa\to +0}\frac{\psifunc{\Uijk{i}{j}{k}+\Lambdaa\xiijk{i}{j}{k}}-\psifunc{\Uijk{i}{j}{k}}}{\Lambdaa}=\psidfunc{\Uijk{i}{j}{k}}\xiijk{i}{j}{k},
    \]
    where
    \[
        \psidfunc{\Uijk{i}{j}{k}}=\begin{dcases}
            3M^{2}\Uijk{i}{j}{k}+2M^{3} & (\Uijk{i}{j}{k}<-M),\\
            \Uijk{i}{j}{k}^{3} & (-M\leq\Uijk{i}{j}{k}\leq M),\\
            3M^{2}\Uijk{i}{j}{k}-2M^{3} & (\Uijk{i}{j}{k}>M).
        \end{dcases}
    \]
    Therefore, since
    \[
        \lim_{\Lambdaa\to +0}\frac{\phionetildefunc{\UU+\Lambdaa\xii}-\phionetildefunc{\UU}}{\Lambdaa}=\sumd{i=0}{\Nx}\sumd{j=0}{\Ny}\sumd{k=0}{\Nz}\psidfunc{\Uijk{i}{j}{k}}\xiijk{i}{j}{k}\Dx\Dy\Dz=\productLdtwo{\psidfunc{\UU}}{\xii},
    \]
    where
    \[
        \left(\psidfunc{\UU}\right)_{i,j,k}\coloneqq\psidfunc{\Uijk{i}{j}{k}},
    \]
    it follows that
    \[
        \nablaLdtwo\phionetildefunc{\UU}=\psidfunc{\UU}.
    \]
    \par We now show that $\phionetilde$ is $3M^{2}$-smooth. We have
    \[
        \normLdtwo{\nablaLdtwo\phionetildefunc{\UU}-\nablaLdtwo\phionetildefunc{\Ubar}}^{2}=\sumd{i=0}{\Nx}\sumd{j=0}{\Ny}\sumd{k=0}{\Nz}\left(\psidfunc{\Uijk{i}{j}{k}}-\psidfunc{\Ubarijk{i}{j}{k}}\right)^{2}\Dx\Dy\Dz.
    \]
    Here,
    \[
        \left(\psidfunc{\Uijk{i}{j}{k}}-\psidfunc{\Ubarijk{i}{j}{k}}\right)^{2}\leq 9M^{4}\left(\Uijk{i}{j}{k}-\Ubarijk{i}{j}{k}\right)^{2}.
    \]
    In fact, if $\Uijk{i}{j}{k}=\Ubarijk{i}{j}{k}$, the claim is trivial; if $\Uijk{i}{j}{k}\neq\Ubarijk{i}{j}{k}$, it suffices to show
    \[
        \begin{split}
            &\left(\psidfunc{\Uijk{i}{j}{k}}-\psidfunc{\Ubarijk{i}{j}{k}}\right)^{2}\leq 9M^{4}\left(\Uijk{i}{j}{k}-\Ubarijk{i}{j}{k}\right)^{2}\\
            \iff&\left|\frac{\psidfunc{\Uijk{i}{j}{k}}-\psidfunc{\Ubarijk{i}{j}{k}}}{\Uijk{i}{j}{k}-\Ubarijk{i}{j}{k}}\right|\leq 3M^{2},
        \end{split}
    \]
    which follows from the mean value theorem and $\left|\psiddfunc{\Uijk{i}{j}{k}}\right|\leq 3M^{2}$, and this inequality can be deduced from
    \[
        \psiddfunc{\Uijk{i}{j}{k}}=\begin{dcases}
            3M^{2} & (\Uijk{i}{j}{k}<-M),\\
            3\Uijk{i}{j}{k}^{2} & (-M\leq\Uijk{i}{j}{k}\leq M),\\
            3M^{2} & (\Uijk{i}{j}{k}>M).
        \end{dcases}
    \]
    Hence,
    \[
        \begin{split}
            \normLdtwo{\nablaLdtwo\phionetildefunc{\UU}-\nablaLdtwo\phionetildefunc{\Ubar}}^{2}&\leq\sumd{i=0}{\Nx}\sumd{j=0}{\Ny}\sumd{k=0}{\Nz}9M^{4}\left(\Uijk{i}{j}{k}-\Ubarijk{i}{j}{k}\right)^{2}\Dx\Dy\Dz\\
            &=\left(3M^{2}\normLdtwo{\UU-\Ubar}\right)^{2}.
        \end{split}
    \]
    This concludes the proof of the $3M^{2}$-smoothness of $\phionetilde$.
    \par Moreover, the convexity of $\phithree$ follows from
    \[
        \begin{split}
            &\quad\productLdtwo{\nablaLdtwo\phithreefunc{\UU}-\nablaLdtwo\phithreefunc{\Ubar}}{\UU-\Ubar}\\
            &=\sumd{i=0}{\Nx}\sumd{j=0}{\Ny}\sumd{k=0}{\Nz}\left(\left|\nabladpfunc{\Uijk{i}{j}{k}-\Ubarijk{i}{j}{k}}\right|^{2}+\left|\nabladmfunc{\Uijk{i}{j}{k}-\Ubarijk{i}{j}{k}}\right|^{2}\right)\Dx\Dy\Dz\\
            &\geq 0,
        \end{split}
    \]
    and the convexity of $\phifour$ can be proved from
    \[
        \productLdtwo{\nablaLdtwo\phifourfunc{\UU}-\nablaLdtwo\phifourfunc{\Ubar}}{\UU-\Ubar}=\varepsilon\normLdtwo{\deltwo\UU-\deltwo\Ubar}^{2}\geq 0,
    \]
    where we used the boundary conditions and Lemma~\ref{lem:integration_by_parts}. Moreover, since
    \[
        \begin{split}
            &\quad\phitwofunc{\UU}-\phitwofunc{\Ubar}-\productLdtwo{\nablaLdtwo\phitwofunc{\Ubar}}{\UU-\Ubar}-\frac{1-\eta}{2}\normLdtwo{\UU-\Ubar}^{2}\\
            &=\sumd{i=0}{\Nx}\sumd{j=0}{\Ny}\sumd{k=0}{\Nz}\frac{1-\eta}{2}\left(\Uijk{i}{j}{k}^{2}-\Ubarijk{i}{j}{k}^{2}-2\Ubarijk{i}{j}{k}\left(\Uijk{i}{j}{k}-\Ubarijk{i}{j}{k}\right)\right.\\
            &\quad\quad\quad\quad\quad\quad\quad\quad\quad\quad\quad\quad\quad\quad\quad\quad\quad\quad\quad\quad\left.-\left(\Uijk{i}{j}{k}-\Ubarijk{i}{j}{k}\right)^{2}\right)\Dx\Dy\Dz\\
            &=0,
        \end{split}
    \]
    if $\eta<1$, then $\phitwo$ is $(1-\eta)$-strongly convex, and if $\eta\geq 1$, then $-\phitwo$ is $(\eta-1)$-strongly convex.
    \par The above discussion yields the desired conclusion.
\end{proof}
In Lemma~\ref{lem:SH_L_convex_mu}, we established the smoothness of $\phionetilde$ instead of $\phione$. Therefore, Theorem~\ref{thm:iwade_general_mu} cannot be applied verbatim to obtain the energy-dissipation law. Nevertheless, unique solvability of the proposed scheme follows as below.
\begin{thm}
    For any $\Dt >0$, the scheme in Definition~\ref{dfn:iwade_SH} admits a unique solution at each time step.
\end{thm}
\begin{proof}
    It follows from Theorem~\ref{thm:iwade_unique} together with Remark~\ref{rem:iwade_unique} and Lemma~\ref{lem:SH_L_convex_mu}.
\end{proof}
For the energy-dissipation law, note that $\phionetildefunc{\Un{n}}$ coincides with $\phionefunc{\Un{n}}$ whenever $\normLd{\infty}{\Un{n}}\leq M$, where $M$ is the parameter in Lemma~\ref{lem:SH_L_convex_mu}. Hence, it suffices to establish an a priori bound on $\normLd{\infty}{\Un{n}}$.
\begin{lem}
    \label{lem:Energy_lower_bound}
    Let
    \[
        \begin{dcases}
            \Hfunc{\UU}\coloneqq\phionefunc{\UU}+\phitwofunc{\UU}-\phithreefunc{\UU}+\phifourfunc{\UU},\\
            \Htilfunc{\UU}\coloneqq\phionetildefunc{\UU}+\phitwofunc{\UU}-\phithreefunc{\UU}+\phifourfunc{\UU}.
        \end{dcases}
    \]
    Then, for arbitrary $\zetaa>\frac{1}{\varepsilon}+\eta-1$,
    \[
        \Hfunc{\UU}\geq\frac{\Ceez}{\CC^{2}}\normLd{\infty}{\UU}^{2}-\frac{\zetaa^{2}}{4}\Lx\Ly\Lz.
    \]
    Here, $\CC$ is a constant which appears in Lemma~\ref{lem:Sobolev_discrete_Laplacian}, and $\Ceez$ is a positive constant
    \[
        \Ceez=\frac{\varepsilon+1-\eta+\zetaa-\sqrt{(\varepsilon-1+\eta-\zetaa)^{2}+4}}{6}.
    \]
    Moreover, if $\zetaa\leq M^{2}$, then for the same constants $\CC,\Ceez$, it follows that
    \[
        \Htilfunc{\UU}\geq\frac{\Ceez}{\CC^{2}}\normLd{\infty}{\UU}^{2}-\frac{\zetaa^{2}}{4}\Lx\Ly\Lz.
    \]
\end{lem}
In the proof of Lemma~\ref{lem:Energy_lower_bound}, we make use of the following lemma, the proof of which is given in the Appendix. 
\begin{lem}
    \label{lem:Sobolev_discrete_Laplacian}
    There exists a constant $\CC$ such that
    \[
        \normLd{\infty}{\UU}\leq\CC\sqrt{\normLdtwo{\UU}^{2}+\normLdtwo{\DD\UU}^{2}+\normLdtwo{\deltwo\UU}^{2}}
    \]
    for all three-dimensional grid functions $\Uijk{i}{j}{k}$ subject to
    \[
        \begin{dcases}
            \deliUU{0}=\deliUU{\Nx}=0,\\
            \deljUU{0}=\deljUU{\Ny}=0,\\
            \delkUU{0}=\delkUU{\Nz}=0,\\
            \delidelijkUU{0}=\delidelijkUU{\Nx}=0,\\
            \deljdelijkUU{0}=\deljdelijkUU{\Ny}=0,\\
            \delkdelijkUU{0}=\delkdelijkUU{\Nz}=0.\\
        \end{dcases}
    \]
\end{lem}
\begin{proof}[The proof of Lemma~\ref{lem:Energy_lower_bound}]
    First, we have $\frac{1}{4}\Uijk{i}{j}{k}^{4}\geq\frac{\zetaa}{2}\Uijk{i}{j}{k}^{2}-\frac{\zetaa^{2}}{4}$, and if $\zetaa\leq M^{2}$, then $\psifunc{\Uijk{i}{j}{k}}\geq\frac{\zetaa}{2}\Uijk{i}{j}{k}^{2}-\frac{\zetaa^{2}}{4}$. Indeed,
    \[
        \frac{1}{4}\Uijk{i}{j}{k}^{4}-\left(\frac{\zetaa}{2}\Uijk{i}{j}{k}^{2}-\frac{\zetaa^{2}}{4}\right)=\left(\frac{1}{2}\Uijk{i}{j}{k}^{2}-\frac{\zetaa}{2}\right)^{2}\geq 0,
    \]
    and
    \[
        \begin{split}
            &\quad\psifunc{\Uijk{i}{j}{k}}-\left(\frac{\zetaa}{2}\Uijk{i}{j}{k}^{2}-\frac{\zetaa^{2}}{4}\right)\\
            &=\begin{dcases}
                \frac{3M^{2}-\zetaa}{2}\Uijk{i}{j}{k}^{2}+2M^{3}\Uijk{i}{j}{k}+\frac{3M^{4}+\zetaa^{2}}{4} & (\Uijk{i}{j}{k}<-M)\\
                \left(\frac{1}{2}\Uijk{i}{j}{k}^{2}-\frac{\zetaa}{2}\right)^{2} & (-M\leq\Uijk{i}{j}{k}\leq M)\\
                \frac{3M^{2}-\zetaa}{2}\Uijk{i}{j}{k}^{2}-2M^{3}\Uijk{i}{j}{k}+\frac{3M^{4}+\zetaa^{2}}{4} & (\Uijk{i}{j}{k}>M)
            \end{dcases}\\
            &\geq\begin{dcases}
                \frac{3M^{2}-\zetaa}{2}(-M)^{2}+2M^{3}(-M)+\frac{3M^{4}+\zetaa^{2}}{4} & (\Uijk{i}{j}{k}<-M)\\
                0 & (-M\leq\Uijk{i}{j}{k}\leq M)\\
                \frac{3M^{2}-\zetaa}{2}M^{2}-2M^{3}\cdot M+\frac{3M^{4}+\zetaa^{2}}{4} & (\Uijk{i}{j}{k}>M)
            \end{dcases}\\
            &\geq 0,
        \end{split}
    \]
    where we used $\frac{2M^{3}}{3M^{2}-\zetaa}\leq\frac{2M^{3}}{3M^{2}-M^{2}}=M$. Hence, we have $\phionefunc{\UU}\geq\frac{\zetaa}{2}\normLdtwo{\UU}^{2}-\frac{\zetaa^{2}}{4}\Lx\Ly\Lz$, and if $\zetaa\leq M^{2}$, then $\phionetildefunc{\UU}\geq\frac{\zetaa}{2}\normLdtwo{\UU}^{2}-\frac{\zetaa^{2}}{4}\Lx\Ly\Lz$.
    \par Moreover, by Lemma~\ref{lem:integration_by_parts} and the boundary conditions,
    \[
        \begin{split}
            \phithreefunc{\UU}&=-\sumd{i=0}{\Nx}\sumd{j=0}{\Ny}\sumd{k=0}{\Nz}\Uijk{i}{j}{k}\deltwo\Uijk{i}{j}{k}\Dx\Dy\Dz\\
            &\leq\sumd{i=0}{\Nx}\sumd{j=0}{\Ny}\sumd{k=0}{\Nz}\left(\frac{1}{2\omegaa}\Uijk{i}{j}{k}^{2}+\frac{\omegaa}{2}\left(\deltwo\Uijk{i}{j}{k}\right)^{2}\right)\Dx\Dy\Dz\\
            &=\frac{1}{2\omegaa}\normLdtwo{\UU}^{2}+\frac{\omegaa}{2}\normLdtwo{\deltwo\UU}^{2}
        \end{split}
    \]
    holds, where $\omegaa$ denotes an arbitrary positive constant.
    \par From the discussion above, it follows that
    \[
        \Hfunc{\UU}\geq\frac{1}{2}\left(\zetaa+1-\eta-\frac{1}{\omegaa}\right)\normLdtwo{\UU}^{2}+\frac{\varepsilon-\omegaa}{2}\normLdtwo{\deltwo\UU}^{2}-\frac{\zetaa^{2}}{4}\Lx\Ly\Lz.
    \]
    Here, let
    \[
        \omegaa=\frac{\sqrt{(\varepsilon-1+\eta-\zetaa)^{2}+4}+\varepsilon-1+\eta-\zetaa}{2}.
    \]
    Then, from
    \[
        \frac{1}{2}\left(\zetaa+1-\eta-\frac{1}{\omegaa}\right)=\frac{\varepsilon-\omegaa}{2}=\frac{3}{2}\Ceez,
    \]
    we obtain
    \[
        \Hfunc{\UU}\geq\frac{3}{2}\Ceez\left(\normLdtwo{\UU}^{2}+\normLdtwo{\deltwo\UU}^{2}\right)-\frac{\zetaa^{2}}{4}\Lx\Ly\Lz.
    \]
    Here, $\Ceez>0$. Indeed, since $\zetaa>\frac{1}{\varepsilon}+\eta-1$, it follows that $1-\eta+\zetaa>\frac{1}{\varepsilon}>0$ and $\varepsilon+1-\eta+\zetaa>0$. Hence,
    \[
        \begin{split}
            \Ceez>0\iff&\varepsilon+1-\eta+\zetaa>\sqrt{(\varepsilon-1+\eta-\zetaa)^{2}+4}\\
            \iff&1-\eta+\zetaa>\frac{1}{\varepsilon}
        \end{split}
    \]
    holds, and it implies that $\Ceez>0$. Using this fact and
    \[
        \normLdtwo{\UU}^{2}+\normLdtwo{\deltwo\UU}^{2}\geq\frac{2}{3}\left(\normLdtwo{\UU}^{2}+\phithreefunc{\UU}+\normLdtwo{\deltwo\UU}^{2}\right),
    \]
    which follows from the previously established $\phithreefunc{\UU}\leq\frac{1}{2\omegaa}\normLdtwo{\UU}^{2}+\frac{\omegaa}{2}\normLdtwo{\deltwo\UU}^{2}$ and the assumption $\omegaa =1$, we obtain
    \[
        \Hfunc{\UU}\geq\Ceez\left(\normLdtwo{\UU}^{2}+\phithreefunc{\UU}+\normLdtwo{\deltwo\UU}^{2}\right)-\frac{\zetaa^{2}}{4}\Lx\Ly\Lz.
    \]
    Therefore, by using Lemma~\ref{lem:Sobolev_discrete_Laplacian}, it follows that
    \[
        \Hfunc{\UU}\geq\frac{\Ceez}{\CC^{2}}\normLd{\infty}{\UU}^{2}-\frac{\zetaa^{2}}{4}\Lx\Ly\Lz.
    \]
    \par The above argument also holds when $\Hfunc{\UU}$ is replaced by $\Htilfunc{\UU}$, assuming $\zetaa\leq M^{2}$.
\end{proof}
From the discussion above, we now prove the energy-dissipation law by proving the boundedness of solutions.
\begin{thm}
    \label{thm:scheme_sup_bound}
    In the scheme of Definition~\ref{dfn:iwade_SH}, for arbitrary $\zetaa>\frac{1}{\varepsilon}+\eta-1$, define the constant $\Ceez$\footnote{It is also worth noting that $\Ceez$ is monotonically increasing with respect to $\zetaa$, and if $\zetaa$ tends to $\infty$, $\Ceez$ approaches $\frac{\varepsilon}{3}$, while if $\zetaa$ tends to 0, $\Ceez$ approaches 0.} which appeared in Lemma~\ref{lem:Energy_lower_bound}, and let
    \[
        M=\max\mleft\{\sqrt{\max\mleft\{0,\zetaa\mright\}},\sqrt{\frac{\CC^{2}}{\Ceez}\left(\Hfunc{\Un{0}}+\frac{\zetaa^{2}}{4}\Lx\Ly\Lz\right)}\mright\}.
    \]
    Then, if $3M^{2}\leq\left|1-\eta\right|$ or $\begin{dcases}
        3M^{2}>\left|1-\eta\right|\\
        \Dt\leq\frac{2}{3M^{2}-\left|1-\eta\right|}
    \end{dcases}$
    holds, $\Hfunc{\Un{n}}$ is non-increasing in $n$, and for any $n$ we have
    \[
        \normLd{\infty}{\Un{n}}\leq\sqrt{\frac{\CC^{2}}{\Ceez}\left(\Hfunc{\Un{0}}+\frac{\zetaa^{2}}{4}\Lx\Ly\Lz\right)}.
    \]
\end{thm}
\begin{proof}
    Let $\Utiln{0}=\Un{0}$, and define $\Utiln{n}$ recursively by
    \[
        \begin{split}
            &\frac{\Utiln{n+1}-\Utiln{n}}{\Dt}\\
            &\quad=\begin{dcases}
                -\left[\nablaLdtwo\phionetildefunc{\Utiln{n}}+\left\{\nablaLdtwo\phitwofunc{\Utiln{n+1}}+\nablaLdtwo\phifourfunc{\Utiln{n+1}}\right\}-\nablaLdtwo\phithreefunc{\Utiln{n}}\right]\\
                \quad\quad\quad\quad\quad\quad\quad\quad\quad\quad\quad\quad\quad\quad\quad\quad\quad\quad\quad\quad\quad\quad\quad\quad\quad\quad\quad\quad\quad\quad\quad(\eta<1),\\
                -\left[\nablaLdtwo\phionetildefunc{\Utiln{n}}+\nablaLdtwo\phifourfunc{\Utiln{n+1}}-\left\{-\nablaLdtwo\phitwofunc{\Utiln{n}}+\nablaLdtwo\phithreefunc{\Utiln{n}}\right\}\right]\\
                \quad\quad\quad\quad\quad\quad\quad\quad\quad\quad\quad\quad\quad\quad\quad\quad\quad\quad\quad\quad\quad\quad\quad\quad\quad\quad\quad\quad\quad\quad\quad(\eta\geq 1).
            \end{dcases}
        \end{split}
    \]
    Suppose that it satisfies the same boundary conditions as in Definition~\ref{dfn:iwade_SH}. Then, by Theorem~\ref{thm:iwade_unique}, there exists a unique sequence $\Utiln{n}$. Moreover, Theorem~\ref{thm:iwade_general_mu} and Lemma~\ref{lem:SH_L_convex_mu} imply that if $3M^{2}\leq\left|1-\eta\right|$ or $\begin{dcases}
        3M^{2}>\left|1-\eta\right|\\
        \Dt\leq\frac{2}{3M^{2}-\left|1-\eta\right|}
    \end{dcases}$
    holds, then $\Htilfunc{\Utiln{n}}$ is non-increasing in $n$.
    \par Because of Lemma~\ref{lem:Energy_lower_bound}, for arbitrary $\zetaa>\frac{1}{\varepsilon}+\eta-1$,
    \[
        \normLd{\infty}{\Utiln{0}}=\normLd{\infty}{\Un{0}}\leq\sqrt{\frac{\CC^{2}}{\Ceez}\left(\Hfunc{\Un{0}}+\frac{\zetaa^{2}}{4}\Lx\Ly\Lz\right)}.
    \]
    Here, if
    \[
        M=\max\mleft\{\sqrt{\max\mleft\{0,\zetaa\mright\}},\sqrt{\frac{\CC}{\Ceez}\left(\Hfunc{\Un{0}}+\frac{\zetaa^{2}}{4}\Lx\Ly\Lz\right)}\mright\},
    \]
    since $\normLd{\infty}{\Utiln{0}}=\normLd{\infty}{\Un{0}}\leq M$, $\Htilfunc{\Utiln{0}}=\Hfunc{\Utiln{0}}=\Hfunc{\Un{0}}$ holds.
    \par We will prove $\Un{n}=\Utiln{n}$ and $\normLd{\infty}{\Un{n}}\leq\sqrt{\frac{\CC^{2}}{\Ceez}\left(\Hfunc{\Un{0}}+\frac{\zetaa^{2}}{4}\Lx\Ly\Lz\right)}$ below by mathematical induction.
    \par The case $n=0$ is clear from the preceding discussion.
    \par We assume that the statement holds for $n-1$ and prove it for $n$. By the induction hypothesis, $\normLd{\infty}{\Utiln{n-1}}=\normLd{\infty}{\Un{n-1}}\leq M$ holds, and it follows that $\nablaLdtwo\phionefunc{\Un{n-1}}=\nablaLdtwo\phionetildefunc{\Utiln{n-1}}$. Hence, the update equations for $\Un{n}$ and $\Utiln{n}$ coincide; hence $\Un{n}=\Utiln{n}$. Moreover, Lemma~\ref{lem:Energy_lower_bound} yields
    \[
        \normLd{\infty}{\Un{n}}=\normLd{\infty}{\Utiln{n}}\leq\sqrt{\frac{\CC^{2}}{\Ceez}\left(\Htilfunc{\Utiln{n}}+\frac{\zetaa^{2}}{4}\Lx\Ly\Lz\right)}.
    \]
    This equation and $\Htilfunc{\Utiln{n}}\leq\Htilfunc{\Utiln{0}}=\Hfunc{\Un{0}}$ yield
    \[
        \normLd{\infty}{\Un{n}}=\normLd{\infty}{\Utiln{n}}\leq\sqrt{\frac{\CC^{2}}{\Ceez}\left(\Hfunc{\Un{0}}+\frac{\zetaa^{2}}{4}\Lx\Ly\Lz\right)}.
    \]
    \par From the above discussion, it is proved that for arbitrary $n$, $\Un{n}=\Utiln{n}$ and $\normLd{\infty}{\Un{n}}\leq\sqrt{\frac{\CC^{2}}{\Ceez}\left(\Hfunc{\Un{0}}+\frac{\zetaa^{2}}{4}\Lx\Ly\Lz\right)}$ hold. The fact that $\Hfunc{\Un{n}}$ is non-increasing in $n$ follows from
    \[
        \begin{dcases}
            \normLd{\infty}{\Un{n-1}}=\normLd{\infty}{\Utiln{n-1}}\leq M,\\
            \normLd{\infty}{\Un{n}}=\normLd{\infty}{\Utiln{n}}\leq M,
        \end{dcases}
    \]
    which implies
    \[
        \Hfunc{\Un{n}}=\Htilfunc{\Utiln{n}}\leq\Htilfunc{\Utiln{n-1}}=\Hfunc{\Un{n-1}}.
    \]
\end{proof}
\subsection{Error Analysis for the Proposed Scheme}
In this subsection, we establish an error estimate for the scheme in Definition~\ref{dfn:iwade_SH}, based on Theorem~\ref{thm:error_iwade}.
\begin{thm}
    Suppose that the exact solution $\uu$ of the Swift--Hohenberg equation satisfies $\uu\in C^{6}$. Assume further that the initial data are consistent, i.e., $\Uijkn{i}{j}{k}{0}=\uu_{0}\mleft(i\Dx,j\Dy,k\Dz\mright)$ for all grid indices $(i,j,k)$. In addition, for any $\zetaa>\frac{1}{\varepsilon}+\eta-1$, define the constant $\Ceez$ as in Lemma~\ref{lem:Energy_lower_bound}, let $M$ be as in Theorem~\ref{thm:scheme_sup_bound}, and set
    \[
        \Md\coloneqq\max\mleft\{M,\sqrt{\frac{\CCd^{2}}{\Ceez}\left(\Htruefunc{\uu_{0}}+\frac{\zetaa^{2}}{4}|\Omega|\right)}\mright\}.
    \]
    Furthermore, assume that $\Lhat\coloneqq 3\Md^{2}-\left|1-\eta\right|\neq 0$. \footnote{This assumption is not essential (cf. Theorem~\ref{thm:error_iwade}). }
    Then, if $\Lhat<0$ or $\begin{cases}
        \Lhat>0\\
        \Dt<\frac{2}{\Lhat}
    \end{cases}$, and if $\Dt$ also satisfies the condition of Theorem~\ref{thm:scheme_sup_bound}, the solution of the numerical scheme converges to the exact solution with rate $\order{\Dt+(\Dx)^{2}+(\Dy)^{2}+(\Dz)^{2}}$ in the meaning of the discrete $L^{2}$ norm.
\end{thm}
\begin{proof}
    Define $\phione,\phitwo,\phithree,\phifour$ as in Corollary~\ref{cor:dfn_phi}. Also, let
    \[
        \psihatfunc{\Uijk{i}{j}{k}}\coloneqq\begin{dcases}
            \frac{3\Md^{2}}{2}\Uijk{i}{j}{k}^{2}+2\Md^{3}\Uijk{i}{j}{k}+\frac{3}{4}\Md^{4} & (\Uijk{i}{j}{k}<-\Md),\\
            \frac{1}{4}\Uijk{i}{j}{k}^{4} & (-\Md\leq\Uijk{i}{j}{k}\leq \Md),\\
            \frac{3\Md^{2}}{2}\Uijk{i}{j}{k}^{2}-2\Md^{3}\Uijk{i}{j}{k}+\frac{3}{4}\Md^{4} & (\Uijk{i}{j}{k}>\Md),
        \end{dcases}
    \]
    and define
    \[
        \phionehatfunc{\UU}\coloneqq\sumd{i=0}{\Nx}\sumd{j=0}{\Ny}\sumd{k=0}{\Nz}\psihatfunc{\Uijk{i}{j}{k}}\Dx\Dy\Dz.
    \]
    Furthermore, let $\Uhatn{n}$ be given by $\Uhatn{0}=\Un{0}$ and
    \[
        \begin{split}
            &\frac{\Uhatn{n+1}-\Uhatn{n}}{\Dt}\\
            &\quad=\begin{dcases}
                -\left[\nablaLdtwo\phionehatfunc{\Uhatn{n}}+\left\{\nablaLdtwo\phitwofunc{\Uhatn{n+1}}+\nablaLdtwo\phifourfunc{\Uhatn{n+1}}\right\}-\nablaLdtwo\phithreefunc{\Uhatn{n}}\right]\\
                \quad\quad\quad\quad\quad\quad\quad\quad\quad\quad\quad\quad\quad\quad\quad\quad\quad\quad\quad\quad\quad\quad\quad\quad\quad\quad\quad\quad\quad\quad\quad(\eta<1),\\
                -\left[\nablaLdtwo\phionehatfunc{\Uhatn{n}}+\nablaLdtwo\phifourfunc{\Uhatn{n+1}}-\left\{-\nablaLdtwo\phitwofunc{\Uhatn{n}}+\nablaLdtwo\phithreefunc{\Uhatn{n}}\right\}\right]\\
                \quad\quad\quad\quad\quad\quad\quad\quad\quad\quad\quad\quad\quad\quad\quad\quad\quad\quad\quad\quad\quad\quad\quad\quad\quad\quad\quad\quad\quad\quad\quad(\eta\geq 1),
            \end{dcases}
        \end{split}
    \]
    with the same boundary conditions as in Definition~\ref{dfn:iwade_SH}. 
    \par Next, define
    \[
        \begin{split}
            \nuone&\coloneqq\phionehat,\\
            \nutwo&\coloneqq\begin{dcases}
                \phitwo+\phifour & (\eta<1),\\
                \phifour & (\eta\geq1),
            \end{dcases}\\
            \nuthree&\coloneqq\begin{dcases}
                \phithree & (\eta<1),\\
                -\phitwo+\phithree & (\eta\geq1),
            \end{dcases}\\
            \nutilonefunc{\uu}&\coloneqq\uu^{3},\\
            \nutiltwofunc{\uu}&\coloneqq\begin{dcases}
                (1-\eta)\uu+\varepsilon\Delta^{2}\uu & (\eta<1),\\
                \varepsilon\Delta^{2}\uu & (\eta\geq1),
            \end{dcases}\\
            \nutilthreefunc{\uu}&\coloneqq\begin{dcases}
                -2\Delta\uu & (\eta<1),\\
                (\eta-1)\uu-2\Delta\uu & (\eta\geq1).
            \end{dcases}
        \end{split}
    \]
    It is clear that $\nutwo,\nuthree,\nutiltwo$ are linear. Moreover, since $\frac{1}{\varepsilon}+\eta-1<\zetaa\leq\Md^{2}$, Lemma~\ref{lem:integration_by_parts} together with the boundary conditions implies
    \[
        \begin{split}
            \nusumfunc{\UU}&\coloneqq\nuonefunc{\UU}+\nutwofunc{\UU}-\nuthreefunc{\UU}\\
            &=\sumd{i=0}{\Nx}\sumd{j=0}{\Ny}\sumd{k=0}{\Nz}\left(\psihatfunc{\Uijk{i}{j}{k}}+\frac{1}{2}\left(1-\eta-\frac{1}{\varepsilon}\right)\Uijk{i}{j}{k}^{2}\right.\\
            &\quad\quad\quad\quad\quad\quad\quad\quad\left.+\left(\sqrt{\frac{\varepsilon}{2}}\deltwo\Uijk{i}{j}{k}+\sqrt{\frac{1}{2\varepsilon}}\Uijk{i}{j}{k}\right)^{2}\right)\Dx\Dy\Dz\\
            &\geq\sumd{i=0}{\Nx}\sumd{j=0}{\Ny}\sumd{k=0}{\Nz}\left(\psihatfunc{\Uijk{i}{j}{k}}-\frac{\Md^{2}}{2}\Uijk{i}{j}{k}^{2}\right)\Dx\Dy\Dz\\
            &\geq-\frac{1}{4}\Md^{4}\Lx\Ly\Lz.
        \end{split}
    \]
    The last inequality follows from
    \[
        \begin{split}
            \psihatfunc{\Uijk{i}{j}{k}}-\frac{\Md^{2}}{2}\Uijk{i}{j}{k}^{2}&=\begin{dcases}
                \Md^{2}\Uijk{i}{j}{k}^{2}+2\Md^{3}\Uijk{i}{j}{k}+\frac{3}{4}\Md^{4} & (\Uijk{i}{j}{k}<-\Md)\\
                \frac{1}{4}\Uijk{i}{j}{k}^{4}-\frac{\Md^{2}}{2}\Uijk{i}{j}{k}^{2} & (-\Md\leq\Uijk{i}{j}{k}\leq\Md)\\
                \Md^{2}\Uijk{i}{j}{k}^{2}-2\Md^{3}\Uijk{i}{j}{k}+\frac{3}{4}\Md^{4} & (\Uijk{i}{j}{k}>\Md)
            \end{dcases}\\
            &\geq-\frac{1}{4}\Md^{4}.
        \end{split}
    \]
    In addition, $\phionehat$ can be shown to be $3\Md^{2}$-smooth in the same way as in Lemma~\ref{lem:SH_L_convex_mu}. Since by assumption there exist some $\rr\in(0,1)$ and positive constants $\kappan{1},\kappan{2}$ satisfying $\kappan{1}+\kappan{2}<2$, and $\Lhat<0$ or $\begin{cases}
        \Lhat>0\\
        \Dt\leq\frac{\rr\kappan{2}}{\Lhat}
    \end{cases}$, Theorem~\ref{thm:error_iwade} together with Lemma~\ref{lem:SH_L_convex_mu} yields the following estimate:
    \[
        \begin{split}
            &\quad\normLdtwo{\Uhatn{n}-\un{n}}\\
            &\leq\sqrt{\frac{\kappan{1}\kappan{2}}{\left(2-\kappan{1}-\kappan{2}\right)\left(\kappan{1}\Lhat^{2}+36\kappan{2}\Md^{4}\right)}}\left(\CT\Dt+\max_{\nd\in\{0,\dots,n-1\}}\normLdtwo{\gamman{\nd+\frac{1}{2}}}\right)\\
            &\quad\quad\quad\quad\quad\quad\quad\quad\quad\quad\quad\quad\times\sqrt{\exp(\left(\frac{36\Md^{4}}{\kappan{1}\Lhat}+\frac{\Lhat}{\kappan{2}\left(1-\rr\thetafunc{\Lhat}\right)}\right)T)-1},
        \end{split}
    \]
    where 
    \[
        \begin{split}
            \gamman{n+\frac{1}{2}}&=\left\{\nablaLdtwo\nuonefunc{\un{n}}+\nablaLdtwo\nutwofunc{\un{n+1}}-\nablaLdtwo\nuthreefunc{\un{n}}\right\}\\
            &\quad-\left\{\iotaafunc{\nutilonefunc{\unn{n}}}+\iotaafunc{\nutiltwofunc{\unn{n+1}}}-\iotaafunc{\nutilthreefunc{\unn{n}}}\right\}\\
            &=\begin{dcases}
                \left\{\nablaLdtwo\phionehatfunc{\un{n}}+\left\{(1-\eta)\un{n+1}+\varepsilon\delfour\un{n+1}\right\}+2\deltwo\un{n}\right\}\\
                \quad\quad-\left\{\un{n}^{3}+(1-\eta)\un{n+1}+\varepsilon\Deltatwoun{n+1}+2\Deltaun{n}\right\} & (\eta<1)\\
                \left\{\nablaLdtwo\phionehatfunc{\un{n}}+\varepsilon\delfour\un{n+1}-\left\{(\eta-1)\un{n}-2\deltwo\un{n}\right\}\right\}\\
                \quad\quad-\left\{\un{n}^{3}+\varepsilon\Deltatwoun{n+1}-(\eta-1)\un{n}+2\Deltaun{n}\right\} & (\eta\geq 1)
            \end{dcases}\\
            &=\left(\nablaLdtwo\phionehatfunc{\un{n}}+\varepsilon\delfour\un{n+1}+2\deltwo\un{n}\right)-\left(\un{n}^{3}+\varepsilon\Deltatwoun{n+1}+2\Deltaun{n}\right).
        \end{split}
    \]
    Here, we extend the definition of $\uijkn{i}{j}{k}{n}$ to $(i\Dx,j\Dy,k\Dz)\notin\Omega$ so that it satisfies the same boundary conditions as in Definition~\ref{dfn:iwade_SH}, i.e.,
    \[
        \begin{dcases}
            \deliu{0}=\deliu{\Nx}=0,\\
            \delju{0}=\delju{\Ny}=0,\\
            \delku{0}=\delku{\Nz}=0,\\
            \delidelijku{0}=\delidelijku{\Nx}=0,\\
            \deljdelijku{0}=\deljdelijku{\Ny}=0,\\
            \delkdelijku{0}=\delkdelijku{\Nz}=0.\\
        \end{dcases}
    \]
    Since Theorem~\ref{thm:scheme_sup_bound} ensures that $\normLd{\infty}{\Un{n}}\leq M\leq\Md$, it follows that $\Un{n}^{3}=\nablaLdtwo\phionehatfunc{\Un{n}}$, namely $\Un{n}=\Uhatn{n}$. In particular, 
    \[
        \begin{split}
            &\quad\normLdtwo{\en{n}}\\
            &\leq\sqrt{\frac{\kappan{1}\kappan{2}}{\left(2-\kappan{1}-\kappan{2}\right)\left(\kappan{1}\Lhat^{2}+36\kappan{2}\Md^{4}\right)}}\left(\CT\Dt+\max_{\nd\in\{0,\dots,n-1\}}\normLdtwo{\gamman{\nd+\frac{1}{2}}}\right)\\
            &\quad\quad\quad\quad\quad\quad\quad\quad\quad\quad\quad\quad\times\sqrt{\exp(\left(\frac{36\Md^{4}}{\kappan{1}\Lhat}+\frac{\Lhat}{\kappan{2}\left(1-\rr\thetafunc{\Lhat}\right)}\right)T)-1}.
        \end{split}
    \]
    Moreover, since
    \[
        \max_{(x,y,z)\in\Omega}\left|u\right|\leq\sqrt{\frac{\CCd^{2}}{\Ceez}\left(\Htruefunc{\uu_{0}}+\frac{\zetaa^{2}}{4}|\Omega|\right)}\leq\Md,
    \]
    it follows that $\nablaLdtwo\phionehatfunc{\un{n}}=\un{n}^{3}$. Consequently,
    \[
        \gamman{n+\frac{1}{2}}=\varepsilon\left(\delfour\un{n+1}-\Deltatwoun{n+1}\right)+2\left(\deltwo\un{n}-\Deltaun{n}\right).
    \]
    By Taylor's theorem, this expression can be written as a linear combination of $(\Dx)^{2}$, $(\Dy)^{2}$, and $(\Dz)^{2}$, whose coefficients depend only on the exact solution and its derivatives up to sixth order. Consequently, there exists a constant $\CTdd$, depending only on the exact solution, such that
    \[
        \begin{split}
            &\quad\normLdtwo{\en{n}}\leq\sqrt{\frac{\kappan{1}\kappan{2}}{\left(2-\kappan{1}-\kappan{2}\right)\left(\kappan{1}\Lhat^{2}+36\kappan{2}\Md^{4}\right)}}\left(\CTdd\left(\Dt+(\Dx)^{2}+(\Dy)^{2}+(\Dz)^{2}\right)\right)\\
            &\quad\quad\quad\quad\quad\quad\quad\quad\quad\quad\quad\quad\times\sqrt{\exp(\left(\frac{36\Md^{4}}{\kappan{1}\Lhat}+\frac{\Lhat}{\kappan{2}\left(1-\rr\thetafunc{\Lhat}\right)}\right)T)-1}.
        \end{split}
    \]
    This completes the proof.
\end{proof}

\subsection{Asymptotic behavior}

Due to the dissipation law (Theorem~\ref{thm:scheme_sup_bound}), the asymptotic behavior of the proposed scheme can be justified by a Lyapunov-type theorem~\cite[Theorem~3.7]{SMSF2015} even when we employ adaptive step sizes. 

First, for the analysis of the asymptotic behavior, we examine the set $ \mathcal{E} $ of fixed points of the semi-discretized ordinary differential equation
\[ \frac{\mathrm{d}}{\mathrm{d}t} U = -U^3 - (1-\eta) U - \varepsilon \delfour U - 2\deltwo U. \]
Note that the zero vector $0$ is an element of $ \mathcal{E} $. 

Since the energy function can be rewritten as 
\[ \phi_1 (U) + \frac{1-\eta - \frac{1}{\varepsilon}}{2} \normLdtwo{U}^2 + \frac{\varepsilon}{2} \productLdtwo{U}{\left(  \deltwo + \frac{1}{\varepsilon} \right)^2 U }, \]
the energy function is strictly convex when $ 1- \eta - \frac{1}{\varepsilon} \ge 0 $. 
Therefore, in this case, $ \mathcal{E} = \{ 0 \} $ holds. 
In other cases, $ \mathcal{E} $ is not so simple.

\begin{thm}
    Suppose that the step sizes are chosen from a closed interval $ \left[ \Delta t_{\mathrm{l}}, \Delta t_{\mathrm{u}} \right] $, where $ \Delta t_{\mathrm{l}} $ and $\Delta t_{\mathrm{u}} $ are the prescribed lower and upper bounds satisfying 
    \[ 0 < \Delta t_{\mathrm{l}} \le \Delta t_{\mathrm{u}} < \begin{cases} \infty & \text{if } 3M^2 \le |1-\eta|, \\ \frac{2}{3M^2 - |1-\eta|} & \text{otherwise.} \end{cases} \]
    ($M$ is the constant defined in Theorem~\ref{thm:scheme_sup_bound}.)
    Then, the sequence $ \{ \Un{n} \}_{n=0}^{\infty} $ of numerical solutions of the proposed scheme possesses at least one accumulation point, and all accumulation points belong to $\mathcal{E}$. 
    In particular, if $ 1- \eta - \frac{1}{\varepsilon} \ge 0 $, then the sequence $ \{ \Un{n} \}_{n=0}^{\infty} $ converges to $0$. 
\end{thm}

\begin{proof}
    It suffices to verify that the assumptions of \cite[Theorem 4.2]{SMSF2015} hold. 
    Since the set $ \{ U \mid \Hfunc{U} \le \Hfunc{\Un{0}} \} $ is compact due to Lemma~\ref{lem:Energy_lower_bound}, the only assumption that needs verification is that 
    $ \Hfunc{\Un{n+1}} = \Hfunc{\Un{n}} $ holds for some $ \Delta t \in \left[ \Delta t_{\mathrm{l}}, \Delta t_{\mathrm{u}} \right] $ only when $ \Un{n} \in \mathcal{E} $. 
    Using Theorem~\ref{thm:iwade_general_mu}, we see
    \[ \Hfunc{\Un{n+1}} - \Hfunc{\Un{n}} = - \left( \frac{1}{\Delta t} - \frac{3M^2 - |1-\eta|}{2} \right) \normLdtwo{ \Un{n+1} - \Un{n} }^2. \]
    The right-hand side vanishes only when $ \Un{n+1} = \Un{n} $, i.e., $ \Un{n} \in \mathcal{E} $. 
\end{proof}


\backmatter


\section*{Statements and Declarations}
\bmhead{Competing Interests} The authors declare they have no competing interests.

\begin{appendices}

\section{Proof of \eqref{equ:u_bound}}
Let $\omegaa$ be an arbitrary positive constant. Then,
\[
    \begin{split}
        \Htruefunc{\uu}&=\int_{\Omega}\left(\frac{1}{4}\uu^{4}+\frac{1-\eta}{2}\uu^{2}+\frac{\varepsilon}{2}(\Delta \uu)^{2}+\uu\Delta\uu\right)\dd{\bm{x}}\\
        &\geq\int_{\Omega}\left(\left(\frac{\zetaa}{2}\uu^{2}-\frac{\zetaa^{2}}{4}\right)+\frac{1-\eta}{2}\uu^{2}+\frac{\varepsilon}{2}(\Delta\uu)^{2}-\left(\frac{1}{2\omegaa}\uu^{2}+\frac{\omegaa}{2}(\Delta\uu)^{2}\right)\right)\dd{\bm{x}}\\
        &=\frac{1}{2}\left(\zetaa+1-\eta-\frac{1}{\omegaa}\right)\int_{\Omega}\uu^{2}\dd{\bm{x}}+\frac{\varepsilon-\omegaa}{2}\int_{\Omega}(\Delta\uu)^{2}\dd{\bm{x}}-\frac{\zetaa^{2}}{4}|\Omega|.
    \end{split}
\]
If we choose
\[
    \omegaa=\frac{\sqrt{(\varepsilon-1+\eta-\zetaa)^{2}+4}+\varepsilon-1+\eta-\zetaa}{2},
\]
then it follows that
\[
    \frac{1}{2}\left(\zetaa+1-\eta-\frac{1}{\omegaa}\right)=\frac{\varepsilon-\omegaa}{2}=\frac{3}{2}\Ceez.
\]
Hence,
\[
    \Htruefunc{\uu}\geq\frac{3}{2}\Ceez\int_{\Omega}\left(\uu^{2}+(\Delta\uu)^{2}\right)\dd{\bm{x}}-\frac{\zetaa^{2}}{4}|\Omega|.
\]
Here, we recall from the proof of Lemma~\ref{lem:Energy_lower_bound} that the constant $\Ceez$ is positive. Using this and noting that
\[
    \int_{\Omega}\left(\uu^{2}+(\Delta\uu)^{2}\right)\dd{\bm{x}}\geq\frac{2}{3}\int_{\Omega}\left(\uu^{2}+|\nabla\uu|^{2}+(\Delta\uu)^{2}\right)\dd{\bm{x}},
\]
which follows from
\[
    \int_{\Omega}|\nabla\uu|^{2}\dd{\bm{x}}=-\int_{\Omega}\uu\Delta\uu\dd{\bm{x}}\leq\frac{1}{2}\int_{\Omega}\left(\uu^{2}+(\Delta\uu)^{2}\right)\dd{\bm{x}},
\]
we deduce
\[
    \Htruefunc{\uu}\geq\Ceez\int_{\Omega}\left(\uu^{2}+|\nabla\uu|^{2}+(\Delta\uu)^{2}\right)\dd{\bm{x}}-\frac{\zetaa^{2}}{4}|\Omega|.
\]
Applying Corollary~9.13 in~\cite{Brezis} and Proposition 7.2 in~\cite{Yellow}, we obtain
\[
    \Htruefunc{\uu}\geq\frac{\Ceez}{\CCd^{2}}\left(\max_{(x,y,z)\in\Omega}\left|\uu\right|\right)^{2}-\frac{\zetaa^{2}}{4}|\Omega|
\]
for some constant $\CCd>0$. Hence,
\[
    \max_{(x,y,z)\in\Omega}\left|\uu\right|\leq\sqrt{\frac{\CCd^{2}}{\Ceez}\left(\Htruefunc{\uu}+\frac{\zetaa^{2}}{4}|\Omega|\right)}\leq\sqrt{\frac{\CCd^{2}}{\Ceez}\left(\Htruefunc{\uu_{0}}+\frac{\zetaa^{2}}{4}|\Omega|\right)}
\]
since $\dv{\Htruefunc{\uu}}{t}\leq 0$.

\section{Proof of Lemma~\ref{lem:Sobolev_discrete_Laplacian}}
We prove Lemma~\ref{lem:Sobolev_discrete_Laplacian}. We begin by fixing notation.
\par We define $\delii$ as the second-order difference operator in the $x$-direction, namely
\[
    \delii\fijk{i}{j}{k}\coloneqq\frac{\fijk{i+1}{j}{k}-2\fijk{i}{j}{k}+\fijk{i-1}{j}{k}}{(\Dx)^{2}},
\]
and similarly define $\deljj$ and $\delkk$ as the corresponding operators in the $y$- and $z$-directions, respectively. In addition, for a three-dimensional grid function $\fijk{i}{j}{k}$ and $\delta,\delta'\in\left\{\delpi,\delpj,\delpk,\delmi,\delmj,\delmk,\delii,\deljj,\delkk\right\}$, define $\normLdtwo{\delta\fF}$ and $\normLdtwo{\delta\delta'\fF}$ by
\[
    \begin{split}
        \normLdtwo{\delta\fF}&\coloneqq\left(\sumd{i=0}{\Nx}\sumd{j=0}{\Ny}\sumd{k=0}{\Nz}\left|\delta\fijk{i}{j}{k}\right|^{2}\Dx\Dy\Dz\right)^{\frac{1}{2}},\\
        \normLdtwo{\delta\delta'\fF}&\coloneqq\left(\sumd{i=0}{\Nx}\sumd{j=0}{\Ny}\sumd{k=0}{\Nz}\left|\delta\delta'\fijk{i}{j}{k}\right|^{2}\Dx\Dy\Dz\right)^{\frac{1}{2}}.
    \end{split}
\]
Moreover, define
\begin{equation}
    \label{equ:norm_DD}
    \begin{split}
        \normLdtwo{\DD^{2}\fF}&\coloneqq\left(\normLdtwo{\delii\fF}^{2}+\normLdtwo{\deljj\fF}^{2}+\normLdtwo{\delkk\fF}^{2}\right.\\
        &\quad\quad+\frac{\normLdtwo{\delpi\delpj\fF}^{2}+\normLdtwo{\delpi\delmj\fF}^{2}+\normLdtwo{\delmi\delpj\fF}^{2}+\normLdtwo{\delmi\delmj\fF}^{2}}{4}\\
        &\quad\quad+\frac{\normLdtwo{\delpj\delpk\fF}^{2}+\normLdtwo{\delpj\delmk\fF}^{2}+\normLdtwo{\delmj\delpk\fF}^{2}+\normLdtwo{\delmj\delmk\fF}^{2}}{4}\\
        &\quad\quad\left.+\frac{\normLdtwo{\delpi\delpk\fF}^{2}+\normLdtwo{\delpi\delmk\fF}^{2}+\normLdtwo{\delmi\delpk\fF}^{2}+\normLdtwo{\delmi\delmk\fF}^{2}}{4}\right)^{\frac{1}{2}}.
    \end{split}
\end{equation}
Define also
\[
    \begin{split}
        \normLdtwo{\nabladp\fF}&\coloneqq\left(\sumd{i=0}{\Nx}\sumd{j=0}{\Ny}\sumd{k=0}{\Nz}\left|\nabladp\fijk{i}{j}{k}\right|^{2}\Dx\Dy\Dz\right)^{\frac{1}{2}},\\
        \normLdtwo{\nabladm\fF}&\coloneqq\left(\sumd{i=0}{\Nx}\sumd{j=0}{\Ny}\sumd{k=0}{\Nz}\left|\nabladm\fijk{i}{j}{k}\right|^{2}\Dx\Dy\Dz\right)^{\frac{1}{2}}.\\
    \end{split}
\]
Furthermore, for any $\delta\in\left\{\delpi,\delpj,\delpk,\delmi,\delmj,\delmk\right\}$, define
\[
    \begin{split}
        \normLdtwo{\nabladp\delta\UU}&\coloneqq\left(\sumd{i=0}{\Nx}\sumd{j=0}{\Ny}\sumd{k=0}{\Nz}\abs{\nabladp\delta\Uijk{i}{j}{k}}^{2}\Dx\Dy\Dz\right)^{\frac{1}{2}},\\
        \normLdtwo{\nabladm\delta\UU}&\coloneqq\left(\sumd{i=0}{\Nx}\sumd{j=0}{\Ny}\sumd{k=0}{\Nz}\abs{\nabladm\delta\Uijk{i}{j}{k}}^{2}\Dx\Dy\Dz\right)^{\frac{1}{2}}.
    \end{split}
\]
\par Lemma~\ref{lem:Sobolev_discrete_Laplacian} follows from Lemma~\ref{lem:norm_of_Laplacian}, which in turn relies on Lemmas~\ref{lem:norm_of_like_Laplacian} and \ref{lem:sumd_i1_i}. The statements and proofs of Lemmas~\ref{lem:norm_of_like_Laplacian} and \ref{lem:norm_of_Laplacian} are adapted from Proposition~7.2 in~\cite{Yellow}.
\par We present the Neumann case; the periodic case follows with minor modifications.
\begin{lem}
    \label{lem:norm_of_like_Laplacian}
    Let $\Uijk{i}{j}{k}$ be a three-dimensional grid function satisfying the boundary conditions
    \begin{equation}
        \label{equ:neumann_BC_U}
        \begin{dcases}
            \deliUU{0}=\deliUU{\Nx}=0,\\
            \deljUU{0}=\deljUU{\Ny}=0,\\
            \delkUU{0}=\delkUU{\Nz}=0,\\
            \delidelijkUU{0}=\delidelijkUU{\Nx}=0,\\
            \deljdelijkUU{0}=\deljdelijkUU{\Ny}=0,\\
            \delkdelijkUU{0}=\delkdelijkUU{\Nz}=0.\\
        \end{dcases}
    \end{equation}
    Then,
    \[
        \begin{dcases}
            \productLdtwo{\deltwo\UU}{\delii\UU}=\frac{\normLdtwo{\nabladp\delpi\UU}^{2}+\normLdtwo{\nabladp\delmi\UU}^{2}+\normLdtwo{\nabladm\delpi\UU}^{2}+\normLdtwo{\nabladm\delmi\UU}^{2}}{4},\\
            \productLdtwo{\deltwo\UU}{\deljj\UU}=\frac{\normLdtwo{\nabladp\delpj\UU}^{2}+\normLdtwo{\nabladp\delmj\UU}^{2}+\normLdtwo{\nabladm\delpj\UU}^{2}+\normLdtwo{\nabladm\delmj\UU}^{2}}{4},\\
            \productLdtwo{\deltwo\UU}{\delkk\UU}=\frac{\normLdtwo{\nabladp\delpk\UU}^{2}+\normLdtwo{\nabladp\delmk\UU}^{2}+\normLdtwo{\nabladm\delpk\UU}^{2}+\normLdtwo{\nabladm\delmk\UU}^{2}}{4}.\\
        \end{dcases}
    \]
\end{lem}
\par In the proof of Lemma~\ref{lem:norm_of_like_Laplacian}, we use the following elementary identity.
\begin{lem}
    \label{lem:sumd_i1_i}
    For any $\bb\in\mathbb{Z}_{>\aaa}$, and any sequence $\fii{i}$,
    \[
        \sumd{i=\aaa}{\bb}\left(\fii{i+1}-\fii{i}\right)=\bigkakko{\frac{\fii{i+1}+\fii{i}}{2}}{i=\aaa}{\bb}
    \]
    holds, where
    \[
        \bigkakko{\fii{i}}{i=\aaa}{\bb}\coloneqq\fii{\bb}-\fii{\aaa}.
    \]
\end{lem}
\begin{proof}[Proof of Lemma~\ref{lem:norm_of_like_Laplacian}]
    By symmetry, it suffices to prove the first identity.
    \par By applying Lemma~\ref{lem:integration_by_parts} while taking the boundary conditions into account, we obtain
    \[
        \begin{split}
            &\quad\frac{\normLdtwo{\nabladp\delpi\UU}^{2}+\normLdtwo{\nabladp\delmi\UU}^{2}+\normLdtwo{\nabladm\delpi\UU}^{2}+\normLdtwo{\nabladm\delmi\UU}^{2}}{4}\\
            &\quad-\productLdtwo{\deltwo\UU}{\delii\UU}\\
            &=\frac{\normLdtwo{\nabladp\delpi\UU}^{2}+\normLdtwo{\nabladp\delmi\UU}^{2}+\normLdtwo{\nabladm\delpi\UU}^{2}+\normLdtwo{\nabladm\delmi\UU}^{2}}{4}\\
            &\quad+\sumd{i=0}{\Nx}\sumd{j=0}{\Ny}\sumd{k=0}{\Nz}\frac{\nabladp\Uijk{i}{j}{k}\cdot\nabladp\delii\Uijk{i}{j}{k}+\nabladm\Uijk{i}{j}{k}\cdot\nabladm\delii\Uijk{i}{j}{k}}{2}\Dx\Dy\Dz\\
            &=\sumd{i=0}{\Nx}\sumd{j=0}{\Ny}\sumd{k=0}{\Nz}\left(\frac{\nabladp\Uijk{i+1}{j}{k}-\nabladp\Uijk{i}{j}{k}}{4\Dx}\cdot\nabladp\delpi\Uijk{i}{j}{k}\right.\\
            &\quad\quad\quad\quad\quad\quad\quad\quad\quad+\frac{\nabladp\Uijk{i}{j}{k}-\nabladp\Uijk{i-1}{j}{k}}{4\Dx}\cdot\nabladp\delmi\Uijk{i}{j}{k}\\
            &\quad\quad\quad\quad\quad\quad\quad\quad\quad+\frac{\nabladm\Uijk{i+1}{j}{k}-\nabladm\Uijk{i}{j}{k}}{4\Dx}\cdot\nabladm\delpi\Uijk{i}{j}{k}\\
            &\quad\quad\quad\quad\quad\quad\quad\quad\quad\left.+\frac{\nabladm\Uijk{i}{j}{k}-\nabladm\Uijk{i-1}{j}{k}}{4\Dx}\cdot\nabladm\delmi\Uijk{i}{j}{k}\right)\Dx\Dy\Dz\\
            &\quad+\sumd{i=0}{\Nx}\sumd{j=0}{\Ny}\sumd{k=0}{\Nz}\left(\nabladp\Uijk{i}{j}{k}\cdot\frac{\nabladp\delpi\Uijk{i}{j}{k}-\nabladp\delpi\Uijk{i-1}{j}{k}}{4\Dx}\right.\\
            &\quad\quad\quad\quad\quad\quad\quad\quad\quad+\nabladp\Uijk{i}{j}{k}\cdot\frac{\nabladp\delmi\Uijk{i+1}{j}{k}-\nabladp\delmi\Uijk{i}{j}{k}}{4\Dx}\\
            &\quad\quad\quad\quad\quad\quad\quad\quad\quad+\nabladm\Uijk{i}{j}{k}\cdot\frac{\nabladm\delpi\Uijk{i}{j}{k}-\nabladm\delpi\Uijk{i-1}{j}{k}}{4\Dx}\\
            &\quad\quad\quad\quad\quad\quad\quad\quad\quad\left.+\nabladm\Uijk{i}{j}{k}\cdot\frac{\nabladm\delmi\Uijk{i+1}{j}{k}-\nabladm\delmi\Uijk{i}{j}{k}}{4\Dx}\right)\Dx\Dy\Dz\\
            &=\sumd{i=0}{\Nx}\sumd{j=0}{\Ny}\sumd{k=0}{\Nz}\left(\frac{\nabladp\Uijk{i+1}{j}{k}\cdot\nabladp\delpi\Uijk{i}{j}{k}-\nabladp\Uijk{i}{j}{k}\cdot\nabladp\delpi\Uijk{i-1}{j}{k}}{4\Dx}\right.\\
            &\quad\quad\quad\quad\quad\quad\quad\quad+\frac{\nabladp\Uijk{i}{j}{k}\cdot\nabladp\delmi\Uijk{i+1}{j}{k}-\nabladp\Uijk{i-1}{j}{k}\cdot\nabladp\delmi\Uijk{i}{j}{k}}{4\Dx}\\
            &\quad\quad\quad\quad\quad\quad\quad\quad+\frac{\nabladm\Uijk{i+1}{j}{k}\cdot\nabladm\delpi\Uijk{i}{j}{k}-\nabladm\Uijk{i}{j}{k}\cdot\nabladm\delpi\Uijk{i-1}{j}{k}}{4\Dx}\\
            &\quad\quad\quad\quad\quad\quad\quad\quad\left.+\frac{\nabladm\Uijk{i}{j}{k}\cdot\nabladm\delmi\Uijk{i+1}{j}{k}-\nabladm\Uijk{i-1}{j}{k}\cdot\nabladm\delmi\Uijk{i}{j}{k}}{4\Dx}\right)\Dx\Dy\Dz\\
        \end{split}
    \]
    \[
        \begin{split}
            &=\sumd{j=0}{\Ny}\sumd{k=0}{\Nz}\left[\frac{\nabladp\Uijk{i+1}{j}{k}\cdot\nabladp\delpi\Uijk{i}{j}{k}+\nabladp\Uijk{i}{j}{k}\cdot\nabladp\delpi\Uijk{i-1}{j}{k}}{8}\right.\\
            &\quad\quad\quad\quad\quad\quad+\frac{\nabladp\Uijk{i}{j}{k}\cdot\nabladp\delmi\Uijk{i+1}{j}{k}+\nabladp\Uijk{i-1}{j}{k}\cdot\nabladp\delmi\Uijk{i}{j}{k}}{8}\\
            &\quad\quad\quad\quad\quad\quad+\frac{\nabladm\Uijk{i+1}{j}{k}\cdot\nabladm\delpi\Uijk{i}{j}{k}+\nabladm\Uijk{i}{j}{k}\cdot\nabladm\delpi\Uijk{i-1}{j}{k}}{8}\\
            &\quad\quad\quad\quad\quad\quad\left.+\frac{\nabladm\Uijk{i}{j}{k}\cdot\nabladm\delmi\Uijk{i+1}{j}{k}+\nabladm\Uijk{i-1}{j}{k}\cdot\nabladm\delmi\Uijk{i}{j}{k}}{8}\right]_{i=0}^{\Nx}\Dy\Dz\\
            &=\sumd{j=0}{\Ny}\sumd{k=0}{\Nz}\left[\frac{\nabladpfunc{\Uijk{i+1}{j}{k}+\Uijk{i}{j}{k}}\cdot\nabladp\delpi\Uijk{i}{j}{k}+\nabladpfunc{\Uijk{i}{j}{k}+\Uijk{i-1}{j}{k}}\cdot\nabladp\delpi\Uijk{i-1}{j}{k}}{8}\right.\\
            &\quad\quad\left.+\frac{\nabladmfunc{\Uijk{i+1}{j}{k}+\Uijk{i}{j}{k}}\cdot\nabladm\delpi\Uijk{i}{j}{k}+\nabladmfunc{\Uijk{i}{j}{k}+\Uijk{i-1}{j}{k}}\cdot\nabladm\delpi\Uijk{i-1}{j}{k}}{8}\right]_{i=0}^{\Nx}\\
            &\quad\quad\quad\quad\quad\quad\quad\quad\quad\quad\quad\quad\quad\quad\quad\quad\quad\quad\quad\quad\quad\quad\quad\quad\quad\quad\quad\quad\quad\quad\quad\quad\quad\quad\quad\Dy\Dz.\\
        \end{split}
    \]
    In fact, this expression vanishes. Indeed, because of the boundary conditions, $\Uijk{\Nx+1}{j}{k}=\Uijk{\Nx-1}{j}{k}$ and $\Uijk{\Nx+2}{j}{k}=\Uijk{\Nx-2}{j}{k}$ hold. Hence, we have
    \[
        \begin{split}
            &\quad\frac{\nabladpfunc{\Uijk{\Nx+1}{j}{k}+\Uijk{\Nx}{j}{k}}\cdot\nabladp\delpi\Uijk{\Nx}{j}{k}+\nabladpfunc{\Uijk{\Nx}{j}{k}+\Uijk{\Nx-1}{j}{k}}\cdot\nabladp\delpi\Uijk{\Nx-1}{j}{k}}{8}\\
            &+\frac{\nabladmfunc{\Uijk{\Nx+1}{j}{k}+\Uijk{\Nx}{j}{k}}\cdot\nabladm\delpi\Uijk{\Nx}{j}{k}+\nabladmfunc{\Uijk{\Nx}{j}{k}+\Uijk{\Nx-1}{j}{k}}\cdot\nabladm\delpi\Uijk{\Nx-1}{j}{k}}{8}\\
            &=\frac{\delpifunc{\Uijk{\Nx+1}{j}{k}+\Uijk{\Nx}{j}{k}}\cdot\delpi\delpi\Uijk{\Nx}{j}{k}+\delpifunc{\Uijk{\Nx}{j}{k}+\Uijk{\Nx-1}{j}{k}}\cdot\delpi\delpi\Uijk{\Nx-1}{j}{k}}{8}\\
            &+\frac{\delpjfunc{\Uijk{\Nx+1}{j}{k}+\Uijk{\Nx}{j}{k}}\cdot\delpj\delpi\Uijk{\Nx}{j}{k}+\delpjfunc{\Uijk{\Nx}{j}{k}+\Uijk{\Nx-1}{j}{k}}\cdot\delpj\delpi\Uijk{\Nx-1}{j}{k}}{8}\\
            &+\frac{\delpkfunc{\Uijk{\Nx+1}{j}{k}+\Uijk{\Nx}{j}{k}}\cdot\delpk\delpi\Uijk{\Nx}{j}{k}+\delpkfunc{\Uijk{\Nx}{j}{k}+\Uijk{\Nx-1}{j}{k}}\cdot\delpk\delpi\Uijk{\Nx-1}{j}{k}}{8}\\
            &+\frac{\delmifunc{\Uijk{\Nx+1}{j}{k}+\Uijk{\Nx}{j}{k}}\cdot\delmi\delpi\Uijk{\Nx}{j}{k}+\delmifunc{\Uijk{\Nx}{j}{k}+\Uijk{\Nx-1}{j}{k}}\cdot\delmi\delpi\Uijk{\Nx-1}{j}{k}}{8}\\
            &+\frac{\delmjfunc{\Uijk{\Nx+1}{j}{k}+\Uijk{\Nx}{j}{k}}\cdot\delmj\delpi\Uijk{\Nx}{j}{k}+\delmjfunc{\Uijk{\Nx}{j}{k}+\Uijk{\Nx-1}{j}{k}}\cdot\delmj\delpi\Uijk{\Nx-1}{j}{k}}{8}\\
            &+\frac{\delmkfunc{\Uijk{\Nx+1}{j}{k}+\Uijk{\Nx}{j}{k}}\cdot\delmk\delpi\Uijk{\Nx}{j}{k}+\delmkfunc{\Uijk{\Nx}{j}{k}+\Uijk{\Nx-1}{j}{k}}\cdot\delmk\delpi\Uijk{\Nx-1}{j}{k}}{8}\\
        \end{split}
    \]
    \[
        \begin{split}
            &=\frac{-2\deli\Uijk{\Nx-1}{j}{k}\cdot\delii\Uijk{\Nx-1}{j}{k}}{8}\\
            &+\frac{-\delpjfunc{\Uijk{\Nx-1}{j}{k}+\Uijk{\Nx}{j}{k}}\cdot\delpj\delpi\Uijk{\Nx-1}{j}{k}+\delpjfunc{\Uijk{\Nx}{j}{k}+\Uijk{\Nx-1}{j}{k}}\cdot\delpj\delpi\Uijk{\Nx-1}{j}{k}}{8}\\
            &+\frac{-\delpkfunc{\Uijk{\Nx-1}{j}{k}+\Uijk{\Nx}{j}{k}}\cdot\delpk\delpi\Uijk{\Nx-1}{j}{k}+\delpkfunc{\Uijk{\Nx}{j}{k}+\Uijk{\Nx-1}{j}{k}}\cdot\delpk\delpi\Uijk{\Nx-1}{j}{k}}{8}\\
            &+\frac{2\deli\Uijk{\Nx-1}{j}{k}\cdot\delii\Uijk{\Nx-1}{j}{k}}{8}\\
            &+\frac{-\delmjfunc{\Uijk{\Nx-1}{j}{k}+\Uijk{\Nx}{j}{k}}\cdot\delmj\delpi\Uijk{\Nx-1}{j}{k}+\delmjfunc{\Uijk{\Nx}{j}{k}+\Uijk{\Nx-1}{j}{k}}\cdot\delmj\delpi\Uijk{\Nx-1}{j}{k}}{8}\\
            &+\frac{-\delmkfunc{\Uijk{\Nx-1}{j}{k}+\Uijk{\Nx}{j}{k}}\cdot\delmk\delpi\Uijk{\Nx-1}{j}{k}+\delmkfunc{\Uijk{\Nx}{j}{k}+\Uijk{\Nx-1}{j}{k}}\cdot\delmk\delpi\Uijk{\Nx-1}{j}{k}}{8}\\
            &=0
        \end{split}
    \]
    and similarly
    \[
        \begin{split}
            &\quad\frac{\nabladpfunc{\Uijk{1}{j}{k}+\Uijk{0}{j}{k}}\cdot\nabladp\delpi\Uijk{0}{j}{k}+\nabladpfunc{\Uijk{0}{j}{k}+\Uijk{-1}{j}{k}}\cdot\nabladp\delpi\Uijk{-1}{j}{k}}{8}\\
            &\quad+\frac{\nabladmfunc{\Uijk{1}{j}{k}+\Uijk{0}{j}{k}}\cdot\nabladm\delpi\Uijk{0}{j}{k}+\nabladmfunc{\Uijk{0}{j}{k}+\Uijk{-1}{j}{k}}\cdot\nabladm\delpi\Uijk{-1}{j}{k}}{8}=0.
        \end{split}
    \]
    This proves
    \[
        \productLdtwo{\deltwo\UU}{\delii\UU}=\frac{\normLdtwo{\nabladp\delpi\UU}^{2}+\normLdtwo{\nabladp\delmi\UU}^{2}+\normLdtwo{\nabladm\delpi\UU}^{2}+\normLdtwo{\nabladm\delmi\UU}^{2}}{4}.
    \]
\end{proof}
\begin{lem}
    \label{lem:norm_of_Laplacian}
    Let $\Uijk{i}{j}{k}$ be a three-dimensional grid function subject to the Neumann boundary conditions \eqref{equ:neumann_BC_U}. Then,
    \[
        \begin{split}
            \normLdtwo{\deltwo\UU}^{2}&=\normLdtwo{\delii\UU}^{2}+\normLdtwo{\deljj\UU}^{2}+\normLdtwo{\delkk\UU}^{2}\\
            &\quad+\frac{\normLdtwo{\delpi\delpj\UU}^{2}+\normLdtwo{\delpi\delmj\UU}^{2}+\normLdtwo{\delmi\delpj\UU}^{2}+\normLdtwo{\delmi\delmj\UU}^{2}}{2}\\
            &\quad+\frac{\normLdtwo{\delpj\delpk\UU}^{2}+\normLdtwo{\delpj\delmk\UU}^{2}+\normLdtwo{\delmj\delpk\UU}^{2}+\normLdtwo{\delmj\delmk\UU}^{2}}{2}\\
            &\quad+\frac{\normLdtwo{\delpi\delpk\UU}^{2}+\normLdtwo{\delpi\delmk\UU}^{2}+\normLdtwo{\delmi\delpk\UU}^{2}+\normLdtwo{\delmi\delmk\UU}^{2}}{2}.
        \end{split}
    \]
    In particular, the norms satisfy
    \[
        \normLdtwo{\DD^{2}\UU}\leq\normLdtwo{\deltwo\UU}\leq\sqrt{2}\normLdtwo{\DD^{2}\UU}.
    \]
\end{lem}
\begin{proof}
    From Lemma~\ref{lem:norm_of_like_Laplacian}, it follows that
    \[
        \begin{split}
            &\quad\normLdtwo{\deltwo\UU}^{2}\\
            &=\frac{\normLdtwo{\nabladp\delpi\UU}^{2}+\normLdtwo{\nabladp\delmi\UU}^{2}+\normLdtwo{\nabladm\delpi\UU}^{2}+\normLdtwo{\nabladm\delmi\UU}^{2}}{4}\\
            &\quad+\frac{\normLdtwo{\nabladp\delpj\UU}^{2}+\normLdtwo{\nabladp\delmj\UU}^{2}+\normLdtwo{\nabladm\delpj\UU}^{2}+\normLdtwo{\nabladm\delmj\UU}^{2}}{4}\\
            &\quad+\frac{\normLdtwo{\nabladp\delpk\UU}^{2}+\normLdtwo{\nabladp\delmk\UU}^{2}+\normLdtwo{\nabladm\delpk\UU}^{2}+\normLdtwo{\nabladm\delmk\UU}^{2}}{4}\\
            &=\sumd{i=0}{\Nx}\sumd{j=0}{\Ny}\sumd{k=0}{\Nz}\frac{\abs{\delii\Uijk{i+1}{j}{k}}^{2}+2\abs{\delii\Uijk{i}{j}{k}}^{2}+\abs{\delii\Uijk{i-1}{j}{k}}^{2}}{4}\Dx\Dy\Dz\\
            &\quad+\sumd{i=0}{\Nx}\sumd{j=0}{\Ny}\sumd{k=0}{\Nz}\frac{\abs{\deljj\Uijk{i}{j+1}{k}}^{2}+2\abs{\deljj\Uijk{i}{j}{k}}^{2}+\abs{\deljj\Uijk{i}{j-1}{k}}^{2}}{4}\Dx\Dy\Dz\\
            &\quad+\sumd{i=0}{\Nx}\sumd{j=0}{\Ny}\sumd{k=0}{\Nz}\frac{\abs{\delkk\Uijk{i}{j}{k+1}}^{2}+2\abs{\delkk\Uijk{i}{j}{k}}^{2}+\abs{\delkk\Uijk{i}{j}{k-1}}^{2}}{4}\Dx\Dy\Dz\\
            &\quad+\sumd{i=0}{\Nx}\sumd{j=0}{\Ny}\sumd{k=0}{\Nz}\frac{\abs{\delpi\delpj\Uijk{i}{j}{k}}^{2}+\abs{\delpi\delmj\Uijk{i}{j}{k}}^{2}+\abs{\delmi\delpj\Uijk{i}{j}{k}}^{2}+\abs{\delmi\delmj\Uijk{i}{j}{k}}^{2}}{2}\Dx\Dy\Dz\\
            &\quad+\sumd{i=0}{\Nx}\sumd{j=0}{\Ny}\sumd{k=0}{\Nz}\frac{\abs{\delpj\delpk\Uijk{i}{j}{k}}^{2}+\abs{\delpj\delmk\Uijk{i}{j}{k}}^{2}+\abs{\delmj\delpk\Uijk{i}{j}{k}}^{2}+\abs{\delmj\delmk\Uijk{i}{j}{k}}^{2}}{2}\Dx\Dy\Dz\\
            &\quad+\sumd{i=0}{\Nx}\sumd{j=0}{\Ny}\sumd{k=0}{\Nz}\frac{\abs{\delpi\delpk\Uijk{i}{j}{k}}^{2}+\abs{\delpi\delmk\Uijk{i}{j}{k}}^{2}+\abs{\delmi\delpk\Uijk{i}{j}{k}}^{2}+\abs{\delmi\delmk\Uijk{i}{j}{k}}^{2}}{2}\Dx\Dy\Dz\\
            &=\normLdtwo{\delii\UU}^{2}+\normLdtwo{\deljj\UU}^{2}+\normLdtwo{\delkk\UU}^{2}\\
            &\quad+\frac{\normLdtwo{\delpi\delpj\UU}^{2}+\normLdtwo{\delpi\delmj\UU}^{2}+\normLdtwo{\delmi\delpj\UU}^{2}+\normLdtwo{\delmi\delmj\UU}^{2}}{2}\\
            &\quad+\frac{\normLdtwo{\delpj\delpk\UU}^{2}+\normLdtwo{\delpj\delmk\UU}^{2}+\normLdtwo{\delmj\delpk\UU}^{2}+\normLdtwo{\delmj\delmk\UU}^{2}}{2}\\
            &\quad+\frac{\normLdtwo{\delpi\delpk\UU}^{2}+\normLdtwo{\delpi\delmk\UU}^{2}+\normLdtwo{\delmi\delpk\UU}^{2}+\normLdtwo{\delmi\delmk\UU}^{2}}{2},
        \end{split}
    \]
    from which the result follows. Here, in the final step, we used the identity
    \[
        \sumd{i=0}{\Nx}\frac{\abs{\delii\Uijk{i+1}{j}{k}}^{2}+2\abs{\delii\Uijk{i}{j}{k}}^{2}+\abs{\delii\Uijk{i-1}{j}{k}}^{2}}{4}=\sumd{i=0}{\Nx}\abs{\delii\Uijk{i}{j}{k}}^{2},
    \]
    which is a consequence of $\delii\Uijk{-1}{j}{k}=\delii\Uijk{1}{j}{k}$ and $\delii\Uijk{\Nx+1}{j}{k}=\delii\Uijk{\Nx-1}{j}{k}$ obtained from the boundary conditions, and
    \[
        \begin{split}
            \sumd{j=0}{\Ny}\frac{\abs{\deljj\Uijk{i}{j+1}{k}}^{2}+2\abs{\deljj\Uijk{i}{j}{k}}^{2}+\abs{\deljj\Uijk{i}{j-1}{k}}^{2}}{4}&=\sumd{j=0}{\Ny}\abs{\deljj\Uijk{i}{j}{k}}^{2},\\
            \sumd{k=0}{\Nz}\frac{\abs{\delkk\Uijk{i}{j}{k+1}}^{2}+2\abs{\delkk\Uijk{i}{j}{k}}^{2}+\abs{\delkk\Uijk{i}{j}{k-1}}^{2}}{4}&=\sumd{k=0}{\Nz}\abs{\delkk\Uijk{i}{j}{k}}^{2},
        \end{split}
    \]
    which can be derived in the same way.
    \par Comparing with \eqref{equ:norm_DD} yields,
    \[
        \normLdtwo{\DD^{2}\UU}^{2}\leq\normLdtwo{\deltwo\UU}^{2}\leq 2\normLdtwo{\DD^{2}\UU}^{2},
    \]
    i.e.,
    \[
        \normLdtwo{\DD^{2}\UU}\leq\normLdtwo{\deltwo\UU}\leq\sqrt{2}\normLdtwo{\DD^{2}\UU}.
    \]
    This concludes the proof of the relationship between $\normLdtwo{\deltwo\UU}$ and $\normLdtwo{\DD^{2}\UU}$.
\end{proof}
\begin{proof}[Proof of Lemma~\ref{lem:Sobolev_discrete_Laplacian}]
    There exists a constant $\CC''$ such that for all three-dimensional grid functions $\fijk{i}{j}{k}$,
    \[
        \normLd{\infty}{\fijk{i}{j}{k}}\leq\CC''\sqrt{\normLdtwo{\fijk{i}{j}{k}}^{2}+\normLdtwo{\DD\fijk{i}{j}{k}}^{2}+\normLdtwo{\DD^{2}\fijk{i}{j}{k}}^{2}}
    \]
    holds; see, e.g., Corollary~9.13 in~\cite{Brezis} with a standard discretization argument. Hence, Lemma~\ref{lem:norm_of_Laplacian} yields the desired result.
\end{proof}

\end{appendices}

\bibliography{sn-bibliography}

@article{Iwade_memo_2_senko_notSH_dissipative_1,
  author		= "Wise, S. M. and Wang, C. and Lowengrub, J. S.",
  title			= "An energy-stable and convergent finite-difference scheme for the phase field crystal equation",
  journal		= "SIAM J. Numer. Anal.",
  volume		= "47",
  number		= "3",
  pages			= "2269--2288",
  year			= "2009",
  doi			= "10.1137/080738143"
}

@article{SH,
  author		= "Swift, J. and Hohenberg, P. C.",
  title			= "Hydrodynamic fluctuations at the convective instability",
  journal		= "Phys. Rev. A",
  volume		= "15",
  number		= "1",
  pages			= "319--328",
  year			= "1977",
  doi			= "10.1103/PhysRevA.15.319"
}

@article{SH_appli_1,
  author		= "Hariz, A. and Bahloul, L. and Cherbi, L. and Panajotov, K. and Clerc, M. and Ferr\'{e}, M. A. and Kostet, B. and Averlant, E. and Tlidi, M.",
  title			= "{Swift--Hohenberg} equation with third-order dispersion for optical fiber resonators",
  journal		= "Phys. Rev. A",
  volume		= "100",
  number		= "2",
  pages			= "023816",
  year			= "2019",
  doi			= "10.1103/PhysRevA.100.023816"
}

@article{SH_appli_2,
  author		= "Hutt, A. and Atay, F. M.",
  title			= "Analysis of nonlocal neural fields for both general and gamma-distributed connectivities",
  journal		= "Phys. D: Nonlinear Phenom.",
  volume		= "203",
  number		= "1--2",
  pages			= "30--54",
  year			= "2005",
  doi			= "10.1016/j.physd.2005.03.002"
}

@article{senko_1,
  author		= "Vi\~{n}als, J. and Hern\'{a}ndez-Garc\'{i}a, E. and {San Miguel}, M. and Toral, R.",
  title			= "Numerical study of the dynamical aspects of pattern selection in the stochastic {Swift--Hohenberg} equation in one dimension",
  journal		= "Phys. Rev. A",
  volume		= "44",
  number		= "2",
  pages			= "1123--1133",
  year			= "1991",
  doi			= "10.1103/PhysRevA.44.1123"
}

@article{senko_2,
  author		= "Christov, C. I. and Pontes, J.",
  title			= "Numerical scheme for {Swift--Hohenberg} equation with strict implementation of {Lyapunov} functional",
  journal		= "Math. Comput. Model.",
  volume		= "35",
  number		= "1--2",
  pages			= "87--99",
  year			= "2002",
  doi			= "10.1016/S0895-7177(01)00151-0"
}

@article{senko_2_2,
  author		= "Lee, H. G.",
  title			= "An energy stable method for the {Swift--Hohenberg} equation with quadratic--cubic nonlinearity",
  journal		= "Comput. Methods Appl. Mech. Eng.",
  volume		= "343",
  pages			= "40--51",
  year			= "2019",
  doi			= "10.1016/j.cma.2018.08.019"
}

@article{senko_2_3,
  author		= "Qi, L. and Hou, Y.",
  title			= "A Second Order Energy Stable {BDF} Numerical Scheme for the {Swift--Hohenberg} Equation",
  journal		= "J. Sci. Comput.",
  volume		= "88",
  pages			= "74",
  year			= "2021",
  doi			= "10.1007/s10915-021-01593-x"
}

@article{senko_3,
  author		= "G\'{o}mez, H. and Nogueira, X.",
  title			= "A new space-time discretization for the {Swift--Hohenberg} equation that strictly respects the {Lyapunov} functional",
  journal		= "Commun. Nonlinear Sci. Numer. Simul.",
  volume		= "17",
  number		= "12",
  pages			= "4930--4946",
  year			= "2012",
  doi			= "10.1016/j.cnsns.2012.05.018"
}

@article{senko_4,
  author		= "Zhao, X. and Yang, R. and Xue, Z. and Sun, H.",
  title			= "Energy-stable and {${L}^{2}$} norm convergent {BDF}3 scheme for the {Swift--Hohenberg} equation",
  journal		= "Numer. Methods Partial Differ. Equ.",
  volume		= "41",
  number		= "4",
  pages			= "{e}70021",
  year			= "2025",
  doi			= "10.1002/num.70021"
}

@article{senko_5,
  author		= "Yang, J. and Kim, J.",
  title			= "Numerical simulation and analysis of the {Swift--Hohenberg} equation by the stabilized {Lagrange} multiplier approach",
  journal		= "Comput. Appl. Math.",
  volume		= "41",
  number		= "1",
  pages			= "20",
  year			= "2022",
  doi			= "10.1007/s40314-021-01726-w"
}

@article{senko_FE,
  author		= "Qi, L. and Hou, Y.",
  title			= "An energy-stable second-order finite element method for the {Swift--Hohenberg} equation",
  journal		= "Comput. Appl. Math.",
  volume		= "42",
  pages			= "5",
  year			= "2023",
  doi			= "10.1007/s40314-022-02144-2"
}

@article{senko_notSH_FV_1,
  author		= "Eymard, R. and Gallou{\"{e}}t, T. and Herbin, R. and Linke, A.",
  title			= "Finite volume schemes for the biharmonic problem on general meshes",
  journal		= "Math. Comput.",
  volume		= "81",
  number		= "280",
  pages			= "2019--2048",
  year			= "2012",
  doi			= "10.1090/S0025-5718-2012-02608-1"
}

@article{senko_notSH_FV_2,
  author		= "Bailo, R. and Carrillo, J. A. and Kalliadasis, S. and Perez, S. P.",
  title			= "Unconditional Bound-Preserving and Energy-Dissipating Finite-Volume Schemes for the Cahn--Hilliard Equation",
  journal		= "Commun. Comput. Phys.",
  volume		= "34",
  number		= "3",
  pages			= "713--748",
  year			= "2023",
  doi			= "10.4208/cicp.OA-2023-0049"
}

@article{Furihata,
  author		= "Furihata, D.",
  title			= "A stable and conservative finite-difference scheme for the {Cahn--Hilliard} equation",
  journal		= "Numer. Math.",
  volume		= "87",
  pages			= "675--699",
  year			= "2001",
  doi			= "10.1007/PL00005429"
}

@article{Iwade_memo_1,
  author		= "Shen, J. and Yang, X.",
  title			= "Numerical approximations of {Allen--Cahn} and {Cahn--Hilliard} equations",
  journal		= "Discrete Contin. Dyn. Syst. -- Ser. A",
  volume		= "28",
  number		= "4",
  pages			= "1669--1691",
  year			= "2010",
  doi			= "10.3934/dcds.2010.28.1669"
}

@book{Yellow,
  author		= "Taylor, M. E.",
  title			= "Partial Differential Equations I: Basic Theory",
  address		= "New York",
  publisher		= "Springer",
  year			= "1996"
}

@book{Brezis,
  author		= "Brezis, H.",
  title			= "Functional Analysis, Sobolev Spaces and Partial Differential Equations",
  address		= "New York",
  publisher		= "Springer",
  year			= "2010",
  doi			= "10.1007/978-0-387-70914-7"
}

@article {SMSF2015,
    AUTHOR = {Sato, S. and Matsuo, T. and Suzuki, H. and Furihata, D.},
     TITLE = {A {L}yapunov-type theorem for dissipative numerical
              integrators with adaptive time-stepping},
   JOURNAL = {SIAM J. Numer. Anal.},
  FJOURNAL = {SIAM Journal on Numerical Analysis},
    VOLUME = {53},
      YEAR = {2015},
    NUMBER = {6},
     PAGES = {2505--2518},
      ISSN = {0036-1429,1095-7170},
   MRCLASS = {65P10 (65J08)},
  MRNUMBER = {3419887},
MRREVIEWER = {Luigi\ Brugnano},
       DOI = {10.1137/140996719},
       URL = {https://doi.org/10.1137/140996719},
}

@book{iwade_convex,
  author		= "Peypouquet, J.",
  title			= "Convex Optimization in Normed Spaces: Theory, Methods and Examples",
  address		= "Cham",
  publisher		= "Springer",
  year			= "2015",
  doi			= "10.1007/978-3-319-13710-0"
}

@book{bounded_gradient_sequence,
  author		= "Mawhin, J. and Willem, M.",
  title			= "Critical Point Theory and Hamiltonian Systems",
  address		= "New York",
  publisher		= "Springer",
  year			= "1989",
  doi			= "10.1007/978-1-4757-2061-7"
}

@article{CH_collision_2,
  author		= "Li, D. and Quan, C. and Tang, T.",
  title			= "Stability and convergence analysis for the implicit-explicit method to the {Cahn--Hilliard} equation",
  journal		= "Math. Comput.",
  volume		= "91",
  number		= "334",
  pages			= "785--809",
  year			= "2022",
  doi			= "10.1090/mcom/3704"
}

@article{CH_collision_1,
  author		= "Chen, L. Q. and Shen, J.",
  title			= "Applications of semi-implicit {F}ourier-spectral method to phase field equations.",
  journal		= "Comput. Phys. Commun.",
  volume		= "108",
  number		= "2--3",
  pages			= "147--158",
  year			= "1998",
  doi         = "10.1016/S0010-4655(97)00115-X"
}

@article{Barbara,
title = {On efficient semi-implicit auxiliary variable methods for the six-order {Swift--Hohenberg} model},
journal = {J. Comput. Appl. Math.},
volume = {419},
pages = {114730},
year = {2023},
doi = {10.1016/j.cam.2022.114730},
author = {Liu, Z. and Chen, C.},
}

@mastersthesis{Iwade,
  author       = {Iwade, D.},
  title        = {A Dissipative Numerical Method Based on Prox-{DCA}},
  school       = {Department of Mathematical Engineering and Information Physics, School of Engineering, The University of Tokyo},
  type         = {Bachelor's thesis},
  address      = {Tokyo, Japan},
  year         = {2025},
  note         = {In Japanese.}
}


\end{document}